\documentclass[12pt]{article}

\usepackage{amsmath,amsfonts,amssymb,amsthm,graphicx,cite}
\usepackage{amscd}
\usepackage{mathrsfs}
\usepackage[all,cmtip]{xy}

\numberwithin{equation}{section}

\begin{document}

\newtheorem{theorem}{Theorem}[section]
\newtheorem{corollary}[theorem]{Corollary}
\newtheorem{lemma}[theorem]{Lemma}
\newtheorem{proposition}[theorem]{Proposition}

\newcommand{\adiffop}{A$\Delta$O}
\newcommand{\adiffops}{A$\Delta$Os}

\newcommand{\be}{\begin{equation}}
\newcommand{\ee}{\end{equation}}
\newcommand{\bea}{\begin{eqnarray}}
\newcommand{\eea}{\end{eqnarray}}
\newcommand{\sh}{{\rm sh}}
\newcommand{\ch}{{\rm ch}}
\newcommand{\einde}{$\ \ \ \Box$ \vspace{5mm}}
\newcommand{\De}{\Delta}
\newcommand{\de}{\delta}
\newcommand{\Z}{{\mathbb Z}}
\newcommand{\N}{{\mathbb N}}
\newcommand{\C}{{\mathbb C}}
\newcommand{\Cs}{{\mathbb C}^{*}}
\newcommand{\R}{{\mathbb R}}
\newcommand{\Q}{{\mathbb Q}}
\newcommand{\T}{{\mathbb T}}
\newcommand{\re}{{\rm Re}\, }
\newcommand{\im}{{\rm Im}\, }
\newcommand{\cW}{{\cal W}}
\newcommand{\cJ}{{\cal J}}
\newcommand{\cE}{{\cal E}}
\newcommand{\cA}{{\cal A}}
\newcommand{\cR}{{\cal R}}
\newcommand{\cP}{{\cal P}}
\newcommand{\cM}{{\cal M}}
\newcommand{\cN}{{\cal N}}
\newcommand{\cI}{{\cal I}}
\newcommand{\cMs}{{\cal M}^{*}}
\newcommand{\cB}{{\cal B}}
\newcommand{\cD}{{\cal D}}
\newcommand{\cC}{{\cal C}}
\newcommand{\cL}{{\cal L}}
\newcommand{\cF}{{\cal F}}
\newcommand{\cG}{{\cal G}}
\newcommand{\cH}{{\cal H}}
\newcommand{\cO}{{\cal O}}
\newcommand{\cS}{{\cal S}}
\newcommand{\cT}{{\cal T}}
\newcommand{\cU}{{\cal U}}
\newcommand{\cQ}{{\cal Q}}
\newcommand{\cV}{{\cal V}}
\newcommand{\cK}{{\cal K}}
\newcommand{\cZ}{{\cal Z}}
\newcommand{\intR}{\int_{-\infty}^{\infty}}
\newcommand{\intI}{\int_{0}^{\pi/2r}}
\newcommand{\limp}{\lim_{\re x \to \infty}}
\newcommand{\limn}{\lim_{\re x \to -\infty}}
\newcommand{\limpn}{\lim_{|\re x| \to \infty}}
\newcommand{\diag}{{\rm diag}}
\newcommand{\Ln}{{\rm Ln}}
\newcommand{\Arg}{{\rm Arg}}
\newcommand{\LHP}{{\rm LHP}}
\newcommand{\RHP}{{\rm RHP}}
\newcommand{\UHP}{{\rm UHP}}
\newcommand{\Res}{{\rm Res}}
\newcommand{\ep}{\epsilon}
\newcommand{\ga}{\gamma}
\newcommand{\ka}{\kappa}
\newcommand{\sing}{{\rm sing}}
\newcommand{\rE}{{\mathrm E}}
\newcommand{\rF}{{\mathrm F}}

\title{Product formulas  for the relativistic and nonrelativistic conical functions}
\author{Martin Halln\"as \\Department of Mathematical Sciences, \\ Loughborough University, Leicestershire LE11 3TU, UK \\ and \\Simon Ruijsenaars \\ School of Mathematics, \\ University of Leeds, Leeds LS2 9JT, UK}


\maketitle

\centerline{\large\it Dedicated to Masatoshi Noumi on the occasion of his 60th birthday}
\medskip

\begin{abstract}
The conical function and its relativistic generalization can be viewed as eigenfunctions of the reduced 2-particle Hamiltonians of the hyperbolic Calogero-Moser system and its relativistic generalization. We prove new product formulas for these functions. As a consequence, we arrive at explicit diagonalizations of integral operators that commute with the 2-particle Hamiltonians and reduced versions thereof. The kernels of the integral operators are expressed as integrals over products of the eigenfunctions and explicit weight functions. The nonrelativistic limits are controlled by invoking novel uniform limit estimates for the hyperbolic gamma function.
\end{abstract}

\tableofcontents


\section{Introduction}

In this paper we obtain product formulas for the conical function specialization of the Gauss hypergeometric function ${}_2F_1$ and its `relativistic' generalization $\cR(a_+,a_-,b;x,y)$ from~\cite{R11}. The latter can be viewed as an eigenfunction of the Hamiltonian associated to the Calogero-Moser system of relativistic hyperbolic $A_1$ type. This generalizes the well-known fact that in suitable variables the conical function specialization of ${}_2F_1$ is an eigenfunction of the Hamiltonian of the Calogero-Moser system of nonrelativistic hyperbolic $A_1$ type.

Somewhat surprisingly, our product formulas for the conical function (obtained by taking limits of their relativistic generalizations in Theorems~2.4 and~2.5 below) seem to be new. In the context of harmonic analysis, a product formula for the more general Jacobi function is known since a long time, cf.~Koornwinder's survey~\cite{Koo84}. This formula arises from a group translate and encodes a convolution structure. By contrast, our product formulas (in Theorems~4.2 and 4.3) cannot be interpreted in terms of a generalized translate. Rather, they give rise to 1-parameter families of commuting integral operators on~$L^2((0,\infty))$.

The $\cR$-function was defined and studied in \cite{R07} (see also \cite{R11}), as a $5$-variable specialization of the more general $8$-variable `relativistic' generalization $R$ of ${}_2F_1$, introduced in \cite{R94}. The definition of the $R$-function in {\it loc.~cit.} is in terms of a contour integral that generalizes the Barnes representation for ${}_2F_1$. New representations of the $R$-function were later obtained by van de Bult \cite{vdB06} and by van de Bult, Rains and Stokman \cite{BRS07}. By suitable specializations the above results lead to three different representations for the $\cR$-function.

More recently, it has been shown that the $\cR$-function admits five further integral representations that, in contrast to previous representations, involve only four hyperbolic gamma functions~\cite{R11}. We only need one of these, which we proceed to detail.

 First, throughout the paper we choose $a_+$ and $a_-$ positive, and use further parameters
\be
\alpha\equiv 2\pi/a_+a_-,\ \ \ a\equiv (a_++a_-)/2,
\ee
\be
a_s\equiv\min(a_+,a_-),\ \ \ a_l\equiv\max(a_+,a_-).
\ee
  The representation  (3.51) in~\cite{R11} for the $\cR$-function amounts to
\be\label{cRJ}
\cR(b;x,y)=(a_+a_-)^{-1/2}G(2ib-ia)J(b;x,y)\prod_{\de=+,-}G(\de y+ia-ib),
\ee
with $J(b;x,y)$ given by
\be\label{Jrep}
J(b;x,y)=\int_\R dz\frac{G(z+x/2-ib/2)G(z-x/2-ib/2)}{G(z+x/2+ib/2)G(z-x/2+ib/2)}\exp(i\alpha zy).
\ee
Here we take at first $(b,x,y)\in (0,2a)\times\R^2$, and  
$G(z)\equiv G(a_+,a_-;z)$ denotes the hyperbolic gamma function, whose salient features are reviewed in Appendix~A.
(Just as we have done above, we shall suppress the dependence on $a_+$, $a_-$, whenever this is not likely to cause ambiguities.) In particular, it is clear from the reflection equation~\eqref{Grefl}   that $J$ is even in~$x$ and~$y$, while the conjugacy relation~\eqref{Gconj} entails real-valuedness for real arguments.

In this paper we mostly deal with the $J$-function~\eqref{Jrep}, as opposed to the $\cR$-function and further avatars introduced shortly. It naturally arises in the step from $N=1$ to $N=2$ in our recent recursive construction of the arbitrary-$N$ joint eigenfunctions of the hyperbolic relativistic Calogero-Moser system \cite{HR14}, and also equals the function $B(b;x,y)$ given by Eq.~(3.24) in~\cite{R11}.

With a view towards making this paper somewhat more self-contained, we proceed to summarise some key properties of $J(b;x,y)$ and several related functions we have occasion to use. The analyticity properties of the $R$-function are known in great detail from Theorem 2.2 in \cite{R99}. Combining this theorem with \eqref{cRJ} and the definition of~$\cR$ as a specialization of~$R$, it is readily seen that $G(ib-ia)J(b;x,y)$ extends to a function that is meromorphic in $b$, $x$ and $y$, with poles that can only be located on the affine hyperplanes
\be
\pm x=2ia-ib+i(ka_++la_-),\ \ \ k,l\in\N\equiv \{ 0,1,2,\ldots\},
\ee
\be
\pm y=ib+i(ka_++la_-),\ \ \ k,l\in\N.
\ee
Moreover, the pole order  is bounded by the corresponding zero order of the product function
\be
\prod_{\de=+,-}E(\de x+ib-ia)E(\de y-ib+ia),
\ee
where the $E$-function is an entire function related to the hyperbolic gamma function by $G(z)=E(z)/E(-z)$ (see Appendix \ref{AppA} for further details).

A pivotal role in obtaining the product formulas for the $J$-function is played by the explicit evaluation 
\be\label{Jeval}
J(b;ib,v)=\sqrt{a_+a_-}G(ia-2ib)\prod_{\de=+,-}G(\de v-ia+ib),
\ee
which follows from \eqref{cRJ} and Eqs.~(2.13) and (2.20) in \cite{R11}.

 Another crucial ingredient is the asymptotic behavior of the $J$-function for $\re x\to\infty$. This involves a specialization of Theorem 1.2 in \cite{R03II}, which deals with the 4-coupling $BC_1$ case, to the 1-coupling $A_1$ case at hand. To state the relevant result, we introduce the   $c$-function
\be\label{c}
c(b;z)\equiv \frac{G(z+ia-ib)}{G(z+ia)},
\ee 
the phase function
\be\label{phi}
\phi(b)\equiv \exp(i\alpha b(b-2a)/4),
\ee 
and the scattering function
\be\label{u}
u(b;z)\equiv -c(b;z)/c(b;-z).
\ee
For later purposes we mention the involution symmetry
\be\label{spsymm}
\phi(2a-b)=\phi(b),\ \ \ u(2a-b;z)=u(b;z).
\ee
By contrast, the $c$-function and weight function,
\be\label{w}
w(b;z)\equiv 1/c(b;z)c(b;-z),
\ee
are not invariant under this involution, but there is a simple relation between the two distinct weight functions:
\be\label{wsym}
w(b;z)w(2a-b;z)=G(z+ia)^2G(-z+ia)^2.
\ee

Next, we rewrite the $J$-function in terms of an E-function defined by
\bea\label{rE}
\rE(b;x,y) & \equiv & (a_+a_-)^{-1/2}\phi(b)G(ib-ia) \frac{J(b;x,y)}{c(b;x)c(2a-b;y)}
\\ \nonumber
& = & \phi(b)G(ib-ia)G(ia-2ib) \frac{\cR(b;x,y)}{c(b;x)c(b;y)},
\eea
where we used \eqref{cRJ}. This E-function is the $A_1$ specialization of the $BC_1$ $\cE$-function dealt with in Theorem 1.2 of \cite{R03II}, cf.~also (2.36)--(2.38) in~\cite{R11}. Setting 
\be\label{Eas}
\rE_{\rm as}(b;x,y)\equiv  \exp( i\alpha xy/2)-u(b;-y) \exp(- i\alpha xy/2),
\ee
this theorem yields a bound
\be\label{Easb}
|(\rE-\rE_{\rm as})(b;x,y)|<C(b,\de,y,\im x)\exp(-\rho\re x),\ \ \ \rho>0,\ \ \re x>\de>0,
\ee
where $C$ is a positive continuous function on $(0,2a)\times (0,\infty)^2\times\R$. Moreover, specializing to $\im x=0$, it is known that the decay rate~$\rho$ can be chosen equal to   any positive number~$r$ satisfying \eqref{r}. 

Now from the asymptotics \eqref{Gas} of the hyperbolic gamma function it is straightforward to infer a bound
\be\label{cb}
c(b;z)/\phi(b)=\exp(-\alpha b z/2)(1+O(\exp(-r\re z)),\ \ \ \re z\to\infty,
\ee
uniformly on $\im z$-compacts. Thus,
assuming $b\in(0,2a)$ and $y\in(0,\infty)$, the leading asymptotic behaviour of $J(b;x,y)$ for $\re x\to\infty$ is given by the function
\be\label{Jas}
J_{\rm as}(b;x,y)\equiv \sqrt{a_+a_-}G(ia-ib)\exp(-\alpha bx/2)\sum_{\tau=+,-}c(2a-b;\tau y) \exp(\tau i\alpha xy/2).
\ee
More specifically, we deduce from the above
\be\label{Jasb}
|(J-J_{\rm as})(b;x,y)|<C(b,\de,y,\im x)\exp(-(\alpha b/2+\rho)\re x),\ \ \ \rho>0,\ \ \re x>\de>0,
\ee
where $C$ is a positive continuous function on $(0,2a)\times (0,\infty)^2\times\R$ and where for $\im x=0$   the decay rate $\rho$ can be chosen equal to   any positive number $r$ satisfying \eqref{r}. Clearly, by evenness of $J(x,y)$ in~$x$, the asymptotics for $\re x\to-\infty$ is given by $J_{\rm as}(-x,y)$. 

We need to invoke some more features of the $\rE$- and $J$-functions that follow by specialization from results in~\cite{R03II} and~\cite{R03III}, cf.~also Subsection~2.2 in~\cite{R11}. First, the E-function satisfies the self-duality relation
\be
\rE (b;x,y)=\rE (b;y,x),
\ee
and has the symmetry property
\be
\rE (b;x,y)=\rE (2a-b;x,y).
\ee
In view of~\eqref{rE}, this entails
\be\label{Jsym}
J(b;x,y)=G(ia-ib)^2J(2a-b;y,x).
\ee 

 Secondly, the $J$-function is a joint eigenfunction of four independent analytic difference operators (henceforth A$\De$Os), two acting on $x$ and two on $y$. The corresponding analytic difference equations (henceforth A$\De$Es) read
\be\label{cRreig1}
A_\de(b;x)J(b;x,y)=2c_\de(y)J(b;x,y),\ \ \ \de=+,-,
\ee
\be\label{cRreig2}
A_\de(2a-b;y)J(b;x,y)=2c_\de(x)J(b;x,y),\ \ \ \de=+,-,
\ee
with  the A$\De$Os given by
\be
A_\de(b;z)\equiv \frac{s_\de(z-ib)}{s_\de(z)}T_{ia_{-\de}}^z+\frac{s_\de(z+ib)}{s_\de(z)}T_{-ia_{-\de}}^z.
\ee
Here, the translation operators are defined on analytic functions by
\be
(T_c^z f)(z)\equiv f(z-c),\ \ \ c\in\C^*,
\ee
and we are using the notation
\be
s_\de(z)\equiv \sinh(\pi z/a_\de),\ \ c_\de(z)\equiv \cosh(\pi z/a_\de),\ \ e_\de(z)\equiv\exp(\pi z/a_\de),\ \ \ \de=+,-.
\ee

Thirdly, at the end of Section~2 we shall use that the generalized Fourier
transform
\be\label{cF}
\cF(b)
\, :\, \cC\equiv C_0^{\infty}((0,\infty))\subset   L^2((0,\infty)) \to  L^2((0,\infty)),\ \ \ b\in(0,2a),
\ee
defined by
\be\label{cFE}
(\cF(b)\psi)(x)\equiv\left(\frac{1}{2a_+a_-}\right)^{1/2}\int_0^{\infty}\rF(b;x,y)\psi(y)dy,\ \ \ \
\psi\in\cC,
\ee
extends to a unitary operator. The $\rF$-function occurring here is the last cousin of the $\cR$-function we need (besides~$J$ and~$\rE$). It can be defined by
\be\label{rF}
\rF(b;x,y)\equiv G(ia-2ib)G(ib-ia)w(b;x)^{1/2}\cR(b;x,y)w(b;y)^{1/2},\ \ b\in(0,2a),\ \ x,y>0,
\ee
with positive square roots understood. Alternatively, it can be rewritten as
\be
\rF(b;x,y)= \phi(b)^{-1}c(b;x)w(b;x)^{1/2}\rE(b;x,y)c(b;y)w(b;y)^{1/2},
\ee
cf.~\eqref{rE}. Moreover, using the identity~\eqref{wsym}, we see that it is related to~$J$ by
\be\label{rFJ}
\rF(b;x,y)= (a_+a_-)^{-1/2} G(ib-ia)w(b;x)^{1/2}J(b;x,y) w(2a-b;y)^{1/2}.
\ee
Just as the $\rE$-function this function has the symmetry properties
\be\label{rFsym}
\rF(b;x,y)=\rF(b;y,x)=\rF(2a-b;x,y),
\ee
but by contrast to the $\rE$-function it is real-valued. (Note that the quotient of~$\rF$ and~$\rE$ is a phase factor involving a square root of the $u$-function~\eqref{u}.)
Thus the transform $\cF(b)$ is not only unitary, but also self-adjoint (hence involutory); clearly, it also satisfies
\be\label{cFsym}
\cF(b)=\cF(2a-b),\ \ \ b\in(0,2a).
\ee
(The transforms  $\cF(a_-)=\cF(a_+)$ amount to the   sine transform, while the limits $\cF(0)=\cF(2a)$ exist and amount to the cosine transform.)

The transform diagonalizes the Hamiltonians 
\be\label{Ham}
H_\de(b;z)\equiv \left(\frac{s_\de(z-ib)}{s_\de(z)}\right)^{1/2}T_{ia_{-\de}}^z \left(\frac{s_\de(z+ib)}{s_\de(z)}\right)^{1/2}+ \Big( i\to -i\Big),\ \  \   \de=+,-,
\ee
in the sense that on the dense subspace~$\cC$ one has
\be
H_{\pm}(b)\cF(b)=2\cF(b) c_{\pm}(\cdot).
\ee
Note that $\cC$ is a domain of essential self-adjointness for the multiplication operators $2c_{\pm}(\cdot)$. Thus the unitary transform~$\cF(b)$ makes it possible to associate commuting self-adjoint operators on $L^2((0,\infty),dz)$ to the A$\De$Os $H_{\pm}(b;z)$.

We are now in a position to sketch  the results and organization  of this paper. In Section~2 we first clarify the special character of the kernel function
\be\label{K}
\cK(b;x,y,z)\equiv \prod_{\de_1,\de_2,\de_3=+,-}G((\de_1 x+\de_2 y+\de_3 z-ib)/2),
\ee
by proving the kernel identities it satisfies, cf.~Proposition~2.1. Then we focus on the function
\be\label{F}
F(b,v;x,y)\equiv \int_\R dz\, w(b;z)J(b;z,v)\cK(b;x,y,z),\ \ \ x,y\in\R,
\ee
and show it is an eigenfunction of the A$\De$Os $A_{\pm}(b;x)$ with eigenvalue $2c_{\pm}(v)$, cf.~Lemma~2.2. By invariance under $x\leftrightarrow y$, this is also true for $A_{\pm}(b;y)$, and then we can appeal to previous uniqueness and continuity results to obtain
\be\label{FcRcR}
F(b,v;x,y)=\lambda(b,v)J(b;x,v)J(b;y,v),\ \ \ b\in(0,2a),\ \ \ x,y\in\R.
\ee

It is now far from obvious, but true that the proportionality factor $\lambda(b,v)$ equals~$4$, hence yielding our key result 
\be\label{prodForm}
J(b;x,v)J(b;y,v)=\frac{1}{2}\int_0^\infty  w(b;z)J(b;z,v)\cK(b;x,y,z)dz,\ \  b\in(0,2a),\ \ x,y,v\in\R, 
\ee
cf.~Theorem~2.4. The proof of this identity hinges on the explicit evaluation~\eqref{Jeval} and a rather arduous asymptotic analysis carried out in~Lemma~2.3. 

Using the $J$-symmetry \eqref{Jsym}, we obtain a second product formula in Theorem~2.5. We stress that we do not invoke the unitary involution $\cF(b)$ to arrive at these product formulas. However, we can derive some remarkable consequences when we employ $\cF(b)$. They emerge after using \eqref{rFJ} to switch from $J$ to the kernel $\rF(b;x,y)/\sqrt{2a_+a_-}$ of $\cF(b)$, cf.~\eqref{cFE}. 

More specifically, we can completely elucidate the Hilbert space features of the family of integral operators $\cI_z(b)$, $z\ge 0$, given by~\eqref{cIb}. Theorem~2.6 reveals that this is a commuting family of bounded self-adjoint operators on $L^2((0,\infty))$ diagonalized by~$\cF(b)$ as the multiplication operators $2J(b;z,\cdot)$. Moreover, the use of $\cF(b)$ allows us to obtain the striking identity~\eqref{id1}.

The functions $\cR$, $J$, $\rE$ and $\rF$ occurring in Section~2 are related by similarity transformations that give rise to different sets of four analytic difference operators whose joint eigenfunctions they are. (Recall we detailed these A$\De$Os for the $J$-function in \eqref{cRreig1} and \eqref{cRreig2}, and for the $\rF$-function in~\eqref{Ham}.) These objects are the center-of-mass reductions of 2-particle counterparts studied  in Section~3. More precisely, we only employ the 2-particle versions of $J$ and $\rF$ given by
\be\label{Psi2}
\Psi_2(b;x,y)\equiv \exp(i\alpha(x_1+x_2)(y_1+y_2)/2)\Psi(b;x_1-x_2,y_1-y_2),\ \ \Psi=J, \rF,
\ee
(cf.~\eqref{J2J} and \eqref{rF2}), and do not consider the 2-particle A$\De$Os whose joint eigenfunctions they are. The 2-particle counterpart of the kernel function $\cK$ is the function~$\cS_2(b;x,y)$ given by~\eqref{cS}. The first result of Section~3 is that the 2-particle eigenfunction~$J_2(b;z,y)$ is also an eigenfunction of the integral operator whose kernel is the product of $\cS_2(b;x,z)$ and the weight function $w(b;z_1-z_2)$; moreover, the eigenvalue is an explicit product of $G$-functions, cf.~Theorem~3.1. 

The proof of this theorem only involves the product formula~\eqref{prodForm} and the Fourier transform \eqref{Fform}. That is, it does not involve the unitary $\cF(b)$. When we employ the 2-particle version $\cF_2(b)$ of $\cF(b)$ (given by~\eqref{cFE2}), we can reformulate Theorem~3.1 as an explicit diagonalization of the integral operator~$\cI_2(b)$ defined by~\eqref{cI2}, cf.~Theorem~3.2. As a consequence, the kernel of the 2-particle integral operator can be represented by a formula (namely~\eqref{kerform}) that enables us to recover the formula~\eqref{id1} for the kernel of the reduced (`center-of-mass') integral operator~$\cI_z(b)$. The crucial step here is to invoke the definition~\eqref{Jrep} of the $J$-function, so that we come full circle.

In Section~4 we derive nonrelativistic counterparts of the results in Sections~2 and~3.
 The self-duality of the joint eigenfunctions is not preserved by the nonrelativistic limit. Indeed, the conical function and its 2-particle generalization are joint eigenfunctions of a hyperbolic (nonrelativistic) differential operator and a rational (relativistic) difference operator, cf.~\cite{R11}. Accordingly, they admit distinct representations involving either hyperbolic functions or the rational (Euler) gamma function.
 
The same is true for the kernel functions, product formulas  and associated integral operators. The first, hyperbolic type of limit can be handled for the integrands by using the previously known limit formula \eqref{GGlim}. The second type, which leads to the $\Gamma$-function, involves the $G\to \Gamma $ limit encoded in the equations \eqref{defH}--\eqref{Plim}.
 
  To control the limits of the pertinent integrals, however, we need uniform bounds for the two types of limits that are strong enough to invoke the dominated convergence theorem. We obtain such bounds in Appendix~B and Appendix~C. These estimates are new and of independent interest.

\section{Product formulas for $J(b;x,y)$}\label{Sec2}
The main purpose of this section is to obtain the product formula   \eqref{prodForm} for the $J$-function, given by \eqref{Jrep}.
We begin by focusing on the kernel function. 
It is clear from its definition~\eqref{K} that $\cK(b;x,y,z)$ is a meromorphic function of $b,x,y,z$, with poles  located only on the affine hyperplanes (cf.~\eqref{Gpo})
\be\label{Kps}
\de_1 x+\de_2 y+\de_3 z=ib-2ia-2i(ka_++la_-),\ \ \ \de_1,\de_2,\de_3=+,-,\ \ k,l\in\N.
\ee
As demonstrated by the following proposition, it satisfies three independent kernel identities.

\begin{proposition}
Letting $b\in\C$, we have
\be\label{Kids}
A_\de(b;x)\cK(b;x,y,z)=A_\de(b;y)\cK(b;x,y,z)=A_\de(b;z)\cK(b;x,y,z),\ \ \ \de=+,-.
\ee
\end{proposition}

\begin{proof}
Since $\cK(x,y,z)$ is manifestly invariant under any permutation of the variables $(x,y,z)$, it suffices to establish the first equality in \eqref{Kids}. After dividing both the left-hand side and the right-hand side by $\cK(b;x-ia_{-\de},y,z)$, we use the A$\De$Es \eqref{GADeEs} to rewrite the result as
\begin{multline}\label{Id}
\frac{s_\de(x-ib)}{s_\de(x)}+\frac{s_\de(x+ib)}{s_\de(x)}\prod_{\de_2,\de_3=+,-}\frac{c_\de((x+\de_2 y+\de_3 z-ib)/2)}{c_\de((x+\de_2 y+\de_3 z+ib)/2)}\\
=\frac{s_\de(y-ib)}{s_\de(y)}\prod_{\de_3=+,-}\frac{c_\de((x-y+\de_3 z-ib)/2)}{c_\de((x-y+\de_3 z+ib)/2)}+\frac{s_\de(y+ib)}{s_\de(y)}\prod_{\de_3=+,-}\frac{c_\de((x+y+\de_3 z-ib)/2)}{c_\de((x+y+\de_3 z+ib)/2)}.
\end{multline}
It is readily seen that both sides are $2ia_\de$-periodic functions of $x$ with equal limits
\be
e_\de(\mp ib)+e_\de(\mp 3ib),\ \ \ \re x\to\pm\infty,
\ee
and that the residues at the (generically simple) poles $x=0$ and $x=ia_\de$ in the period strip cancel. By Liouville's theorem, it remains to verify that the residues at the poles $x=\de_2 y+\de_3 z-ib+ia_\de$ cancel as well, and this amounts to a routine calculation. (In fact, one need only check the case $\de_2=\de_3=+$, since the left-hand side and right-hand side of \eqref{Id} are even functions of both $y$ and $z$.) 
\end{proof}

Assuming $(b,v)\in(0,2a)\times (0,\infty)$, we show next that the product $J(b;x,v)J(b;y,v)$ can be expressed as an integral over the auxiliary variable $z$ in the integrand
\be\label{I}
I(b,v;x,y,z)\equiv w(b;z)J(b;z,v)\cK(b;x,y,z),
\ee
occurring on the right-hand side of~\eqref{F}.
(Below, we will often suppress the dependence on $b$ and $v$, whenever this is not likely to cause ambiguities.)

Combining \eqref{w} with \eqref{GADeEs} and \eqref{GE}, we deduce
\be
w(z)=\prod_{\de=+,-}2s_\de(z)\frac{E(\de z+ib-ia)}{E(\de z+ia-ib)}.
\ee
Since the product function $J(z,v)\prod_{\de=+,-}E(\de z+ib-ia)$ is holomorphic for all $z\in\C$, it follows that poles of $I(x,y,z)$ are due either to zeros of the $E$-functions $E(\de z+ia-ib)$, $\de=+,-$, which are located at (cf.~\eqref{Ezs})
\be
\de z=ib+i(ka_++la_-),\ \ \ \de=+,-,\ \ k,l\in\N,
\ee
or poles of the kernel function $\cK(x,y,z)$. Letting first $x,y\in\R$, it is clear that the integration contour $\R$ stays away from these pole sequences. Making use of the asymptotics \eqref{Gas} of the hyperbolic gamma function, we readily infer from \eqref{w} and \eqref{K} that
\be\label{was}
w(b;z)=\exp(\alpha b|\re z|)(1+O(\exp(-r|\re z|))),\ \ \ |\re z|\to\infty,
\ee
\be\label{Kas}
\cK(b;x,y,z)=\exp(-\alpha b|\re z|)(1+O(\exp(-r|\re z|))),\ \ \ |\re z|\to\infty,
\ee
where the decay rate is any positive number $r$ satisfying \eqref{r} and the implied constants are uniform for $(b,\im z)$ and $(b,x,y,\im z)$ varying over compact subsets of $(0,2a)\times\R$ and $(0,2a)\times\C^2\times\R$, respectively. Hence \eqref{Jas}--\eqref{Jasb} entail that the integrand $I$ decays exponentially as $|z|\to\infty$, so that the function~\eqref{F}
is well defined.

To motivate the next step we consider the special $b$-values
\be\label{bmn}
b_{mn}\equiv ma_++na_-,\ \ \ m,n\in\Z,
\ee
inasmuch as they satisfy $b_{mn}\in (0,2a)$, and recall from Section 2.3 in \cite{R11} that the vector space of meromorphic joint solutions $f(x)$ to the A$\De$Es
\be\label{feig}
A_\pm(b_{mn};z)f(z)=2c_\pm(v)f(z),\ \ \ v>0,\ \ a_+/a_-\notin\Q,
\ee
is two-dimensional. More specifically, the subspace consisting of even meromorphic solutions is one-dimensional and is thus spanned by the function $J(b_{mn};z,v)$. With this result in mind, we aim to prove that $F(x,y)$ (as initially defined by \eqref{F}) has a meromorphic continuation to all of $\C$ as a function of $x$ which satisfies the A$\De$Es \eqref{feig}. For this, we need not restrict attention to the special $b$-values \eqref{bmn}.

Clearly, the fixed contour representation \eqref{F} and the pole locations \eqref{Kps} for $\cK(x,y,z)$ entail that $F(x,y)$ is holomorphic in the domain
\be
D\equiv \{(x,y)\in\C^2 | |\im x|+|\im y|<2a-b\}.
\ee
For sufficiently small values of $b$, this already suffices to show that $F$ is a joint eigenfunction of the A$\De$Os $A_\de(x)$, $\de=+,-$, with the expected eigenvalues. More precisely, we fix attention on the subset
\be
D_s\equiv \{(x,y)\in D | (x+i\eta,y)\in D, \forall\eta\in[-a_l,a_l]\},\ \ \ b\in(0,a_s),
\ee
for which the shifted arguments remain within the holomorphy domain $D$. Note that the restriction on $b$ is necessary in order for $D_s$ to be non-empty.

\begin{lemma}\label{FeigL}
Fixing $b\in(0,a_s)$ and $v\in(0,\infty)$, and letting $(x,y)\in D_s$, we have the eigenvalue equations
\be\label{Feig}
A_\de(x) F(v;x,y)=2c_\de(v)F(v;x,y),\ \ \ \de=+,-.
\ee
\end{lemma}

\begin{proof}
Thanks to the analyticity of $F$ in $D$ and the restriction to $D_s$, the shifts of $x$ by $\pm ia_{-\de}$ are well defined. Hence we are allowed to act with the A$\De$Os under the integral sign. Using the kernel identities \eqref{Kids} to trade the action of $A_\de(x)$ on $\cK(x,y,z)$ for that of $A_\de(z)$, we can thus rewrite the left-hand side of \eqref{Feig} as
\be
\int_\R dz\ w(z)J(z,v)\left(\frac{s_\de(z-ib)}{s_\de(z)}T_{ia_{-\de}}^z+\frac{s_\de(z+ib)}{s_\de(z)}T_{-ia_{-\de}}^z\right)\cK(x,y,z).
\ee
Note that the pole at $z=0$ due to the denominator $s_\de(z)$ is matched by a double zero of $w(z)$, cf.~\eqref{w} and the locations \eqref{Ezs} of the $G$-zeros.

Changing integration variable, we obtain
\begin{multline}
\int_{\R-ia_{-\de}}dz\ \frac{s_\de(z+ia_{-\de}-ib)}{s_\de(z+ia_{-\de})}w(z+ia_{-\de})J(z+ia_{-\de},v)\cK(x,y,z)\\
+\int_{\R+ia_{-\de}}dz\ \frac{s_\de(z-ia_{-\de}+ib)}{s_\de(z-ia_{-\de})}w(z-ia_{-\de})J(z-ia_{-\de},v)\cK(x,y,z).
\end{multline}
From the difference equations \eqref{GADeEs} satisfied by $G(z)$ it follows that we have
\be
\frac{w(z\pm ia_{-\de})}{w(z)}=\frac{s_\de(z\pm ia_{-\de})}{s_\de(z\pm ia_{-\de}\mp ib)}\frac{s_\de(z\pm ib)}{s_\de(z)},
\ee
so that the above sum of integrals equals
\begin{multline}
\int_{\R-ia_{-\de}}dz\ w(z)\frac{s_\de(z+ib)}{s_\de(z)}J(z+ia_{-\de},v)\cK(x,y,z)\\
+\int_{\R+ia_{-\de}}dz\ w(z)\frac{s_\de(z-ib)}{s_\de(z)}J(z-ia_{-\de},v)\cK(x,y,z).
\end{multline}

We claim that when we shift the contour $\R-ia_{-\de}$ in the former integral up by $a_{-\de}$ and the contour $\R+ia_{-\de}$ in the latter integral down by $a_{-\de}$, then no poles are met. (The asymptotic behaviour of $J$, as given by \eqref{Jas}--\eqref{Jasb}, and the bounds \eqref{was}--\eqref{Kas} ensure that the shift causes no problems at the tail ends.) Taking this claim for granted, the contours of the two integrals are now both equal to $\R$, so that we are entitled to invoke the eigenvalue equation \eqref{cRreig1}, which implies~\eqref{Feig}.

To complete the proof, it remains to verify the claim. Clearly, we remain within the holomorphy domain of both $J$ and $\cK$ while performing the relevant contour shifts. Letting first $a_{-\de}=a_s$, we meet only the simple poles of the factors $1/G(\pm z+ia-ib)$ at $\pm z=ib$, but they are matched by zeros of $s_\de(z\mp ib)$. Likewise, zeros of $G(\pm z+ia)$ match the simple poles of $1/s_\de(z)$ at $z=0$ and at $z=\pm ia_s$, with the latter being met when $a_s=a_l$. Thus, the first/second integrand is regular for $z\in\R - i\eta$ / $z\in\R + i\eta$ and $\eta\in[0,a_s]$.

Finally, consider the case $a_{-\de}=a_l$. Then we encounter simple poles of $1/G(\pm z+ia-ib)$ located at points $\pm z=ib+ika_s$ with $ka_s\leq a_l-b$, but they are again matched by zeros of $s_\de(z\mp ib)$. Moreover, the simple poles of $1/s_\de(z)$ at $\pm z=ka_s$ with $ka_s\leq a_l$ are matched by zeros of $G(\pm z+ia)$, so that the first/second integrand is regular for $z\in\R - i\eta$ / $z\in\R + i\eta$  and $\eta\in[0,a_l]$. Hence, our claim is proved.
\end{proof}

Keeping the assumptions in Lemma \ref{FeigL}, we now fix $y\in\R$. Then it follows that $F(x,y)$ is holomorphic in $x$ for $|\im x|<2a-b$, and that the A$\De$Es \eqref{Feig}, which take the explicit form
\be\label{FeigExp}
\frac{s_\de(x-ib)}{s_\de(x)}F(x-ia_{-\de},y)+\frac{s_\de(x+ib)}{s_\de(x)}F(x+ia_{-\de},y)=2c_\de(v)F(x,y),
\ee
hold true for $|\im x|<a_s-b$. Moreover, by the assumption $b\in (0,a_s)$, we have $2a-b>a_l$ and $a_s-b>0$. A moment's thought reveals that this state of affairs implies that $F(x,y)$ has a meromorphic continuation to all of $\C$ as a function of $x$. Indeed, multiplying \eqref{FeigExp} by $s_\de(x)/s_\de(x-ib)$, we can continue $F(x,y)$ in steps of size $a_{-\de}$ to the lower half plane. Likewise, upon multiplication by $s_\de(x)/s_\de(x+ib)$, we can continue $F(x,y)$ to the upper half plane.

Appealing to the above uniqueness result, we deduce
\be\label{Fx}
F(b_{mn},v;x,y)=\kappa(b_{mn},v,y)J(b_{mn};x,v) ,\ \ \ x,y\in\R,
\ee
under the assumption that the parameters satisfy the conditions
\be\label{parco}
v>0,\ \ \ b_{mn}\in(0,a_s),\ \ \ a_+/a_-\notin\Q.
\ee
Now due to the manifest invariance of $F(x,y)$ under the interchange $x\leftrightarrow y$, it follows that $F(x,y)$ is also meromorphic in $y$ and satisfies the $y$-version of the A$\De$Es~\eqref{FeigExp}. Therefore we also have
\be\label{Fy}
F(b_{mn},v;x,y)=\kappa(b_{mn},v,x)J(b_{mn};y,v) ,\ \ \ x,y\in\R,
\ee
with~\eqref{parco} in force. Inspecting the quotient of~\eqref{Fx} and~\eqref{Fy}, we   
 conclude that we have
\be\label{FcRcRaux}
F(b_{mn},v;x,y)=\lambda(b_{mn},v)J(b_{mn};x,v)J(b_{mn};y,v),\ \ \ x,y\in\R,
\ee
with the function $\lambda(b_{mn},v)$ to be determined. Since the $b$-values $b_{mn}$ are dense in $(0,a_s)$ when $a_+/a_-$ is irrational, (real) analyticity in $b$ now entails 
\eqref{FcRcR}. Moreover, by continuity in $a_\pm$ we can allow any positive $a_\pm$-values in~\eqref{FcRcR}. 

Until further notice, we now assume $b\in(0,a)$. Then we are allowed to set $y=ib$ in \eqref{F} with $x$ real, since this $F$-representation can be analytically continued to $|\im y|<2a-b$. Hence, using the reflection equation \eqref{Grefl} and evenness of the integrand in $z$, we arrive at  
\be\label{Fib2}
F(x,ib)=2\int_0^\infty dz \cI(x,z),\ \ \ x\in\R,\ \ b\in(0,a),
\ee
with the integrand given by
\be\label{cI}
\cI(b,v;x,z)=w(b;z)J(b;z,v)\prod_{\de_1,\de_2=+,-}G((\de_1 x+\de_2 z)/2-ib).
\ee

On the other hand, combining~\eqref{Jeval} and~\eqref{FcRcR} we obtain the special value
\be\label{Fib1}
F(b,v;x,ib)=\sqrt{a_+a_-}G(ia-2ib)\lambda(b,v)J(b;x,v)\prod_{\de=+,-}G(\de v-ia+ib),
\ee
and by \eqref{Jas}--\eqref{Jasb} we know the leading asymptotic behaviour of the right-hand side of \eqref{Fib1} as $x\to\infty$. Therefore, we can determine $\lambda(b,v)$ by computing the $x\to\infty$ asymptotics of the right-hand side of \eqref{Fib2} and comparing it with that of \eqref{Fib1}.

Substituting the right-hand side of \eqref{Gas} for $G(z)$, we find that
\be\label{Gpes}
\prod_{\de=+,-}G(\de t/2-ib)=\exp(-\alpha bt/2)(1+O(\exp(-rt))),\ \ \ t\to\infty,
\ee
where $r$ is any positive number satisfying \eqref{r}.  We now use the bounds~\eqref{was}, \eqref{Jasb} and \eqref{Gpes} with $t=x+z$  in a telescoping argument to deduce that we have
\begin{multline}\label{intgas}
w(z)J(z,v)\prod_{\de=+,-}G(\de(x+z)/2-ib)\\
=\exp(-\alpha bx/2)\big(\exp(\alpha bz/2)J_{\rm as}(z,v)+O(\exp(-rz))\big),
\end{multline}
for $z\to\infty$. (Recall that in \eqref{Jasb} we can choose the decay rate $\rho$ equal to $r$ in case $\im x=0$.) It follows, in particular, that the  left-hand side of this equality is bounded for $z\in(0,\infty)$. 

Choosing instead $t=x-z$ in \eqref{Gpes}, we 
find that the product of the remaining $G$-factors in the integrand $\cI$ decays exponentially as $x\to\infty$ and $z$ varies over a compact subset of $(0,\infty)$. This state of affairs suggests that when we substitute \eqref{intgas} in $\cI$, then the first term should yield the leading asympotic behaviour of the integral in \eqref{Fib2}. We proceed to make this suggestion precise.

\begin{lemma}
Assuming $(b,v)\in(0,a)\times (0,\infty)$, we have
\begin{multline}\label{Fas}
F(b,v;x,ib)=2\exp(-\alpha bx/2)\\
\times\left[\int_\R dz \exp(\alpha bz/2)J_{\rm as}(b;z,v)\prod_{\de=+,-}G(\de(x-z)/2-ib)+O(\exp(-r_m x/2))\right], 
\end{multline}
for $x\to\infty$, with decay rate
\be\label{rm}
r_m=\min(\alpha b/2,r).
\ee
\end{lemma}

\begin{proof}
Taking $x\in(0,\infty)$ from now on, we multiply~\eqref{intgas} by $\prod_{\de=+,-}G(\de(x-z)/2-ib)$ and integrate $z$ over $(x/2,\infty)$. Since this multiplier function is bounded, we deduce that we have 
\begin{multline}\label{crux}
\int_{x/2}^\infty dz\cI(x,z)=\exp(-\alpha bx/2)\\
\times\left[\int_{x/2}^\infty dz \exp(\alpha bz/2)J_{\rm as}(z,v)\prod_{\de=+,-}G(\de(x-z)/2-ib)+O(\exp(-rx/2))\right],
\end{multline}
for $x\to\infty$. 

The point is now that we want to arrive at the integral over all of $\R$ in~\eqref{Fas}, since it can be evaluated explicitly (as we shall presently show, cf.~\eqref{Gint1}--\eqref{Gint2}). In view of~\eqref{Fib2}, we first need to add the integral of~$\cI(x,z)$ over $(0,x/2)$ to the left-hand side of~\eqref{crux}, so as to obtain $F(x,ib)/2$. To estimate this addition, we use~\eqref{intgas} as before. But now we conclude from boundedness of the function $\exp(\alpha b z/2)J_{\rm as}(z,v)$ for $z>0$ (cf.~\eqref{Jas}) that there exists a positive constant $C$ such that
\be
\left|\int_0^{x/2}dz\cI(x,z)\right|<C\exp(-\alpha bx/2)\int_0^{x/2}dz\prod_{\de=+,-}G(\de(x-z)/2-ib),\ \ \ x\in(0,\infty).
\ee
(Note that the $G$-product in \eqref{Gpes} is positive for $t\in\R$, cf.~the conjugacy relation~\eqref{Gconj} and the locations \eqref{Ezs} of the $G$-zeros.)

Next, we invoke~\eqref{Gpes} with $t=x-z$ to deduce
\be\label{Gintb}
\int_\xi^{x/2}dz\prod_{\de=+,-}G(\de(x-z)/2-ib)=O(\exp(-\alpha bx/4)),\ \ \ x\to\infty,
\ee
where the implied constant is uniform for $\xi$ varying over any fixed interval of the form $(-\infty,c)$, $c\in\R$. In particular, for $\xi=0$ we arrive at
\be
\exp(\alpha bx/2)\int_0^{x/2}dz\cI(x,z)=O(\exp(-\alpha bx/4)),\ \ \ x\to\infty,
\ee
but the bound~\eqref{Gintb} also entails
\be
\int_{-\infty}^{x/2} dz \exp(\alpha bz/2)J_{\rm as}(z,u)\prod_{\de=+,-}G(\de(x-z)/2-ib)=O(\exp(-\alpha bx/4)), 
\ee
as $x\to\infty$. From this the lemma readily follows. 
\end{proof}

Taking $z\to x+2z$ and keeping \eqref{Jas} in mind, a straightforward computation reveals that the integral in \eqref{Fas} equals
\be\label{Gint1}
2\exp(\alpha bx/2)J_{\rm as}(x,v)\int_\R dz\exp(i\alpha zv)\prod_{\de=+,-}G(\de z-ib).
\ee
An explicit evaluation of the latter integral is readily obtained from a special case of a Fourier transform formula in \cite{R11}, as reviewed in Appendix \ref{AppA}. More specifically, setting $\nu=-\mu=ib$ in \eqref{Fform}, we obtain
\be\label{Gint2}
\int_\R dz\exp(i\alpha zv)\prod_{\de=+,-}G(\de z-ib)=\sqrt{a_+a_-}G(ia-2ib)\prod_{\de=+,-}G(\de v-ia+ib).
\ee
Substituting these results into \eqref{Fas}, we deduce 
\be
F(b,v;x,ib)\sim 4\sqrt{a_+a_-}G(ia-2ib)J_{\rm as}(b;x,v)\prod_{\de=+,-}G(\de v-ia+ib),\ \ \ x\to\infty.
\ee
Comparing this expression for the leading asymptotic behaviour of $F(b,v;x,ib)$ as $x\to\infty$ with that given by \eqref{Fib1} upon substituting $J_{\rm as}$ for $J$, we find that
\be
\lambda(b,v)=4,\ \ \ b\in(0,a),\ \ v>0.
\ee
By (real) analyticity this equality immediately extends to all $b\in(0,2a)$. Since the integrand \eqref{I} in \eqref{F} is an even function of $z$, we have thus established the following result.

\begin{theorem}
Letting $b\in(0,2a)$  and $ x,y,v\in\R$, we have
\be\label{prform}
J(b;x,v)J(b;y,v)=\frac{1}{2}\int_0^\infty dz\, w(b;z)J(b;z,v)\prod_{\de_1,\de_2,\de_3=+,-}G((\de_1 x+\de_2 y+\de_3 z-ib)/2).
\ee
\end{theorem}

We proceed to obtain a second product formula that looks quite different at face value.

\begin{theorem}
Letting $b\in(0,2a)$  and $ x,t,u\in\R$, we have
\bea\label{prformalt}
J(b;x,t)J(b;x,u) & = & \frac{1}{2}G(ia-ib)^2\int_0^\infty dv\,  w(2a-b;v)J(b;x,v)
\\ \nonumber
& & \times \prod_{\de_1,\de_2,\de_3=+,-}G((\de_1 t+\de_2 u+\de_3 v+ib)/2-ia).
\eea
\end{theorem}
\begin{proof}
We take $b\to 2a-b$ in~\eqref{prform} and then use the symmetry relation~\eqref{Jsym}. Relabeling variables, we see that~\eqref{prformalt} results.
\end{proof}

It is of interest to specialize these formulas to the free cases $b=a_{\pm}$. From~Eq.~(3.4) in~\cite{HR14} we readily obtain
\be
J(a_{\de};x,\hat{x})=\frac{a_{-\de}\sin (\alpha x\hat{x}/2)}{2s_{-\de}(x)s_{\de}(\hat{x})},\ \ \ \de=+,-.
\ee
Also, using~\eqref{GADeEs} we get from~\eqref{w} and~\eqref{c} the free weight functions
\be
w(a_{\de};z)=4s_{-\de}(z)^2,\ \ \ \de=+,-.
\ee
Taking $b=a_{+}$ in \eqref{prform}, we deduce
\be\label{prodap}
\int_0^\infty dz \frac{s_-(z)\sin(\alpha zv/2)}{\prod_{\de,\de'=+,-}c_-((z+\de x+\de' y)/2)}=\frac{4a_-}{s_+(v)}\cdot \frac{\sin(\alpha xv/2)}{s_-(x)}\cdot \frac{\sin(\alpha yv/2)}{s_-(y)}.
\ee
Furthermore, when we take $b=a_+$ in \eqref{prformalt} and use $G(i(a_- -a_+)/2)^2=a_-/a_+$, then we obtain~\eqref{prodap} with $a_+$ and $a_-$ swapped.

It seems that even this elementary hyperbolic product formula is a new result. With hindsight, however, it can be obtained from a limit of the elliptic product formula in Theorem~2.2 of~\cite{R13}. (To verify this, the limit formulas Eqs.~(2.92), (3.129) and (3.131) in~\cite{R97} can be used.)

We proceed to detail another perspective on the above product formulas. This arises when we rewrite them in terms of the unitary transform kernel $\rF(b;x,y)/(2a_+a_-)^{1/2}$, cf.~\eqref{cF}--\eqref{rFsym}. Let us define a family of integral operators $\cI_z(b)$ on $L^2((0,\infty))$  by
\be\label{cIb}
(\cI_z(b)f)(x)\equiv \int_0^{\infty} dy\, w(b;x)^{1/2}\cK(b;x,y,z)w(b;y)^{1/2}f(y),\ \ \ z\ge 0,\ \ \ b\in(0,2a),
\ee
where $x>0$ and positive square roots are taken. It follows from the $G$-asymptotics \eqref{Gas} that for fixed $(b,x,z)\in (0,2a)\times [0,\infty)^2$ the function $\cK(b;x,y,z)w(b;y)^{1/2}$ has exponential decay for $y\to \infty$, so the integral is absolutely convergent for any $f\in L^2((0,\infty))$. At face value, however, it is not clear that the right-hand side of \eqref{cIb} yields a function in $L^2((0,\infty),dx)$. 
Furthermore, when we fix~$b$,  there appears to be no reason for all of the operators $\cI_z(b)$, $z\ge 0$, to commute. 

Even so, more is true: They are bounded self-adjoint operators satisfying
\be\label{cIcom}
[\cI_u(b), \cI_v(b)]=0,\ \  [ \cI_u(b),  \cI_v(2a-b)]=0,\ \ \ u,v\ge 0.
\ee
This is immediate from the following result, which shows that the two families are simultaneously diagonalized by the unitary involution~$\cF(b)$, yielding bounded real-valued multiplication operators.

\begin{theorem}
Let $b\in(0,2a)$ and $z\ge 0$. Then the operator~$\cI_z(b)$ is bounded, and the operator $\cF(b)\cI_z(b)\cF(b)$ on $L^2((0,\infty),dv)$ acts as multiplication by the function $2J(b;z,v)$. Moreover,
 the operator $\cF(b)\cI_z(2a-b)\cF(b)$ on $L^2((0,\infty),dv)$ acts as multiplication by $2J(2a-b;z,v)$.
\end{theorem}
\begin{proof}
When we swap $z$ and $y$ in~\eqref{prform} and use~\eqref{rFJ}, we obtain
\be\label{cIeig}
2J(b;z,v)\frac{\rF(b;x,v)}{\sqrt{2a_+a_-}}=\int_0^{\infty}dy \, w(b;x)^{1/2}\cK(b;x,y,z)w(b;y)^{1/2}\frac{\rF(b;y,v)}{\sqrt{2a_+a_-}}.
\ee 
Integrating this with $f(v)\in C_0^{\infty}((0,\infty))$ we obtain
\be\label{cFf}
2\cF(b)J(b;z,\cdot)f=\cI_z(b)\cF(b)f.
\ee 
From this we read off the first assertion. Using~\eqref{cFsym}, we obtain the second one. 
\end{proof}

Yet another illuminating equivalent version of the product formula~\eqref{prform} is obtained when we multiply~\eqref{cIeig} by  $\rF(b;t,v)/\sqrt{2a_+a_-}$ and integrate over~$v$. This results in
\be
\frac{1}{a_+a_-}\int_0^{\infty}dvJ(b;z,v)\rF(b;x,v)\rF(b;t,v)=w(b;x)^{1/2}\cK(b;x,t,z)w(b;t)^{1/2}.
\ee
Using \eqref{rFJ}, this can be rewritten as the identity
\begin{multline}\label{id1}
w(b;x)^{1/2}w(b;y)^{1/2}w(b;z)^{1/2}\cK(b;x,y,z)=
\\
\frac{G(ia-ib)}{\sqrt{a_+a_-}}\int_0^{\infty}\frac{dv}{w(2a-b;v)^{1/2}}\rF(b;x,v)\rF(b;y,v)\rF(b;z,v).
\end{multline}
(Alternatively, this identity follows when we replace~$f$ in~\eqref{cFf} by $\cF(b)f$.) From the $b\to 2a-b$ invariance \eqref{rFsym} of the $\rF$-function, we also obtain
\begin{multline}\label{id2}
w(2a-b;x)^{1/2}w(2a-b;y)^{1/2}w(2a-b;z)^{1/2}\cK(2a-b;x,y,z)=
\\
\frac{G(ib-ia)}{\sqrt{a_+a_-}}\int_0^{\infty}\frac{dv}{w(b;v)^{1/2}}\rF(b;x,v)\rF(b;y,v)\rF(b;z,v).
\end{multline}




\section{An application to the hyperbolic relativistic \\ Calogero-Moser 2-particle system}
In the recent paper \cite{HR14} we developed a recursive scheme to construct joint eigenfunctions for the commuting A$\De$Os associated with the integrable $N$-particle systems of hyperbolic relativistic Calogero-Moser type. In this section we establish a remarkable application of the product formula \eqref{prform} from Theorem~2.4  to the $N=2$ case of this recursive scheme. More specifically, we  show that the joint eigenfunction of the A$\De$Os is also an eigenfunction of an explicit integral operator, with the eigenvalues being explicit as well.

The kernel of the pertinent integral operator is the product of the weight function $w(b;y_1-y_2)$ and the special function
\be\label{cS}
\cS_2(b;x,y)\equiv \prod_{j,k=1}^2 \frac{G(x_j-y_k-ib/2)}{G(x_j-y_k+ib/2)},
\ee
which connects the $2$-particle A$\De$Os $A_{k,\de}^{(2)}(x)$ to themselves via kernel identities, see (2.2) in \cite{HR14}. For the first step $N=1\to N=2$ of the recursive scheme, however, the main protagonist is a kernel function connecting these $N=2$ A$\De$Os to the elementary $N=1$ A$\De$Os $A_{1,\de}^{(1)}(-y_1)\equiv\exp(ia_{-\de}\partial_{y_1})$. The latter arises from $\cS_2$ by first multiplying by a suitable plane wave and then letting $y_2$ go to infinity, yielding
\be
\cS_2^\sharp(b;x,y_1)\equiv \prod_{j=1}^2 \frac{G(x_j-y_1-ib/2)}{G(x_j-y_1+ib/2)}.
\ee

Starting the recursion with the plane wave
\be
J_1(x_1,y_1)\equiv \exp(i\alpha x_1y_1),
\ee
the first step $N=1\to N=2$ of the recursive scheme yields the function
\be\label{J2}
J_2(b;x,y)=\exp(i\alpha y_2(x_1+x_2))\int_\R dz I_2(b;x,y,z),\ \ \ (b,x,y)\in(0,2a)\times\R^2\times\R^2,
\ee
with integrand
\be
I_2(b;x,y,z)\equiv \cS_2^\sharp(b;x,z)J_1(z,y_1-y_2).
\ee
We now recall from Section 4 in \cite{HR14} the relation of~$J_2$ and~$J$. It arises by taking $z\to z+(x_1+x_2)/2$ in the integral in \eqref{J2} and comparing the result with the formula~\eqref{Jrep} for~$J$: 
\be\label{J2J}
J_2(b;(x_1,x_2),(y_1,y_2))=\exp(i\alpha(x_1+x_2)(y_1+y_2)/2)J(b;x_1-x_2,y_1-y_2).
\ee

We are now ready to formulate and prove the integral equation for~$J_2$.

\begin{theorem}
Letting $(b,x,y)\in (0,2a)\times\R^2\times \R^2$, we have the integral equation
\be\label{intEq}
\int_{\R^2} dz\, w(b;z_1-z_2)\cS_2(b;x,z)J_2(b;z,y)=2\mu(b;y)J_2(b;x,y),
\ee
where
\be\label{defmu}
\mu(b;y)\equiv a_+a_-G(ia-ib)^2\prod_{j=1}^2\prod_{\de=+,-}G(\de y_j+ib/2-ia).
\ee
\end{theorem}

\begin{proof}
First, we aim to rewrite the left-hand side of \eqref{intEq} in such a way that the product formula \eqref{prform} from Theorem~2.4 may be invoked. Substituting the right-hand side of~\eqref{J2J} for~$J_2$ and changing integration variables to
\be\label{zs}
z\equiv z_1-z_2,\ \ \ s\equiv z_1+z_2,
\ee
we use the reflection equation \eqref{Grefl} to rewrite the left-hand side of~\eqref{intEq} as
\begin{multline}
\frac{1}{2}\int_\R ds \exp(i\alpha s(y_1+y_2)/2)\\
\times\int_\R dz w(z) J(z,y_1-y_2)\prod_{\de_1,\de_2,\de_3=+,-}G\big((\de_1(x_1-x_2)+\de_2 (s-(x_1+x_2))+\de_3 z-ib)/2\big).
\end{multline}
Taking $s\to s+x_1+x_2$ and using evenness of the latter integrand in $z$, we deduce that this equals
\begin{multline}
\exp(i\alpha (x_1+x_2)(y_1+y_2)/2)\int_\R ds \exp(i\alpha s(y_1+y_2)/2)\\
\times\int_0^\infty dz w(z)J(z,y_1-y_2)\cK(x_1-x_2,s,z),
\end{multline}
where, just as in Section \ref{Sec2}, the kernel function $\cK$ is given by \eqref{K}. Invoking the product formula \eqref{prform} as well as~\eqref{J2J}, we obtain
\be
2J_2(x,y)\int_\R ds \exp(i\alpha s(y_1+y_2)/2)J(s,y_1-y_2).
\ee

Next, we show that the representation \eqref{Jrep} and the Fourier transform formula \eqref{Fform} allow us to compute the remaining integral explicitly. Changing integration variable according to $z\to z/2$ in \eqref{Jrep}, and once more making use of the reflection equation \eqref{Grefl}, we find that the integral is given by
\be
\frac{1}{2}\int_\R ds\int_\R dz \exp\big(i\alpha(s(y_1+y_2)+z(y_1-y_2))/2\big)\prod_{\de_1,\de_2=+,-}G((\de_1 z+\de_2 s-ib)/2).
\ee
Reversing the change of variables \eqref{zs}, we arrive at the product of one-variable integrals
\be
\prod_{j=1}^2 \int_\R dz_j \exp(i\alpha z_jy_j)\prod_{\de=+,-}G(\de z_j-ib/2).
\ee
Finally, setting $\nu=-\mu=ib/2$ in \eqref{Fform}, we obtain an evaluation formula for these integrals, which yields the right-hand side of \eqref{intEq}.
\end{proof}

Our next theorem is a simple corollary of Theorem~3.1.

\begin{theorem}
Letting $(b,x,y)\in (0,2a)\times\R^2\times \R^2$, we have  
\be\label{intEqalt}
\int_{\R^2} dz\, w(2a-b;z_1-z_2)\cS_2(2a-b;x,z)J_2(b;y,z)=2\mu(2a-b;y)J_2(b;y,x).
\ee
\end{theorem}
\begin{proof}
This follows from~\eqref{intEq} by taking $b\to 2a-b$ and using~\eqref{Jsym}.
\end{proof}

Next, we define
\be\label{rF2}
\rF_2(b;x,y)\equiv \exp(i\alpha(x_1+x_2)(y_1+y_2)/2)\rF(b;x_1-x_2,y_1-y_2),\ \ x,y\in G_2,
\ee
where we have introduced 
\be\label{G2}
G_2\equiv \{ x\in\R^2\mid x_2<x_1\}.
\ee
The integrands on the left-hand sides of~\eqref{intEq} and~\eqref{intEqalt} are invariant under swapping $z_1$ and~$z_2$, so we can use~\eqref{rFJ} to rewrite these formulas in terms of~$\rF_2$. This yields
\be\label{int1}
\int_{G^2} dz\, w(b;x_1-x_2)^{1/2}\cS_2(b;x,z)w(b;z_1-z_2)^{1/2}\rF_2(b;z,y)=\mu(b;y)\rF_2(b;x,y),\ x,y\in G_2,
\ee
and
\begin{multline}\label{int2}
\int_{G^2} dz\, w(2a-b;x_1-x_2)^{1/2}\cS_2(2a-b;x,z)w(2a-b;z_1-z_2)^{1/2}\rF_2(b;z,y)
\\
=\mu(2a-b;y)\rF_2(b;x,y),\ \ \ x,y\in G_2.
\end{multline}

Now the generalized Fourier transform
\be\label{cF2}
\cF_2(b)
\, :\, \cC_2\equiv C_0^{\infty}(G_2)\subset   L^2(G_2) \to  L^2(G_2),\ \ \ b\in(0,2a), 
\ee
defined by
\be\label{cFE2}
(\cF_2(b)\psi)(x)\equiv \frac{1}{a_+a_-} \int_{G_2}\rF_2(b;x,y)\psi(y)dy,\ \ \ \
\psi\in\cC_2,\ \ x\in G_2,
\ee
extends to a unitary operator. (This readily follows from the tensor product structure exhibited by the kernel $\rF_2$.) Just as in Section~2, we proceed to consider the integral operator given by
\be\label{cI2}
(\cI_2(b)f)(x)\equiv \int_{G_2} dy\,  w(b;x_1-x_2)^{1/2}\cS_2(b;x,y)w(b;y_1-y_2)^{1/2}f(y),\ f\in L^2(G_2), \ b\in(0,2a).
\ee
In view of~\eqref{Grefl} and~\eqref{Gconj}, its kernel is positive and invariant under swapping~$x$ and~$y$. As in Section~2, however, it is not obvious that $\cI_2(b)$  is bounded, let alone that it satisfies 
\be
[\cI_2(b),\cI_2(2a-b)]=0.
\ee
These properties are a direct consequence of our next result.

\begin{theorem}
Letting $b\in(0,2a)$, the operator $\cI_2(b)$ is bounded, and the operator $\cF_2(b)^*\cI_2(b)\cF_2(b)$ on $L^2(G_2,dv)$ acts as multiplication by the positive function $\mu(b;v)$ given by \eqref{defmu}. Moreover,
 the operator $\cF_2(b)^*\cI_2(2a-b)\cF_2(b)$ on $L^2(G_2,dv)$ acts as multiplication by~$\mu(2a-b;v)$.
\end{theorem}
\begin{proof}
This follows as before from \eqref{int1} and \eqref{int2}.
\end{proof}

 Taking~$y\to v$ in \eqref{int1} and then integrating with $\overline{\rF_2}(b;y,v)$, we obtain the identity
\be\label{iden1}
w(b;x_1-x_2)^{1/2}\cS_2(b;x,y)w(b;y_1-y_2)^{1/2}
=\frac{1}{(a_+a_-)^2}\int_{G_2}dv\, \mu(b;v)\rF_2(b;x,v)\overline{\rF_2}(b;y,v).
\ee
Since $\rF_2$ is invariant under $b\to 2a-b$, this implies a second identity that we shall not spell out. 

We conclude this section by clarifying the relation of these identities to \eqref{id1} and \eqref{id2}. When we transform~\eqref{iden1} to sum and difference variables, we obtain
\begin{multline}\label{kerform}
w(b;x)^{1/2}\cK(b;x,y, s-t)w(b;y)^{1/2}
=\frac{1}{2(a_+a_-)^2}
\\
\times \int_{\R}dr \int_0^{\infty}dv\mu(b;(r+v)/2,(r-v)/2)\exp(i\alpha(r(s-t)/2)\rF(b;x,v)\rF(b;y,v).
\end{multline}
Setting $z:=s-t$, we can now combine the definitions~\eqref{defmu} and~\eqref{Jrep} of~$\mu$ and~$J$ to get
\be
\int_{\R} dr\mu(b;(r+v)/2,(r-v)/2)\exp(i\alpha rz/2)=2a_+a_- G(ia-ib)^2 J(2a-b;v,z).
\ee
Using next the $J$-symmetry \eqref{Jsym}, the right-hand side becomes $2a_+a_- J(b;z,v)$. Finally, trading~$J$ for~$\rF$ by using~\eqref{rFJ}, we arrive at \eqref{id1}.
  
\newpage

\section{Nonrelativistic limit formulas}

Subsection 4.2 of~\cite{R11} deals with the nonrelativistic limit of the $\cR$-function. In terms of $J(b;x,y)$~\eqref{Jrep}, the starting point is the reparametrized $J$-function
\be\label{Jbeta}
J(\pi,\beta,\beta g;r,\beta k)=\frac{1}{2}\int_{\R}dt\prod_{\de =+,-}\frac{G(\pi,\beta;(t+\de r-i\beta g)/2)}{G(\pi,\beta;(t+\de r+i\beta g)/2)}\exp(itk),
\ee
and the limit amounts to taking $\beta$ to 0.  When we formally interchange it with the integration, we can use~\eqref{GGlim}. This yields
\be\label{JFlim} 
\lim_{\beta\to 0}J(\pi,\beta,\beta g;r,\beta k)=  \frac{1}{2}\int_{\R} dt\, \frac{ \exp(itk)}{\prod_{\de=+,-}[2\cosh((t+\de r)/2)]^{g}}.
\ee
Taking $r$ and $k$ real, the latter integral obviously  converges for $\re g>0$. It equals the function $F(g;r,2k)$ defined by Eq.~(65) of our joint paper~\cite{HR15}. Its relation to the conical (or Mehler) function $P_{ik-1/2}^{1/2-g}(\cosh r)$ is given by
\be
F(g;r,2k)=\left(\frac{\pi}{4}\right)^{1/2}\frac{\Gamma(g +ik) \Gamma(g -ik)}{\Gamma(g)(2\sinh r)^{g -1/2}}
P_{ik-1/2}^{1/2-g}(\cosh r),
\ee
cf.~Eq.~14.12.4 in~\cite{Dig10}. In addition, one can produce a number of expressions for $F$ in terms of the hypergeometric function ${}_2F_1$. For example, the conical function specialization of Eq.~14.3.15 in \cite{Dig10} and the duplication formula for the gamma function entail
\be\label{FHypergRep}
F(g;r,2k)=\frac{\Gamma(g+ik)\Gamma(g-ik)}{2\Gamma(2g)}{}_2F_1(g+ik,g-ik;g+1/2;-\sinh^2(r/2)).
\ee

Now in~\cite{R11} the above interchange of limits was left uncontrolled. With Prop.~B.1 at our disposal, we can not only remedy this, but also obtain bounds on the exponential decay in~$r$ and~$ k$ of~$J(\pi,\beta,\beta g;r,\beta k)$ that are uniform for~$\beta$ small enough. This is detailed in the following proposition.

\begin{proposition}
Let $(g,r,k)\in\C^3$ be restricted by
\be
(g,r,k)\in \{\re g>0\}\times \{|\im r| <\pi/2\}\times\{|\im k|<\re g\}=:\cP.
\ee
Then we have
\be\label{JF}
\lim_{\beta\to 0}J(\pi,\beta,\beta g;r,\beta k)= F(g;r,2k),
\ee
where the limit is uniform on compacts of the parameter space~$\cP$. Next, suppose $(g,r,k)\in(0,\infty)\times\R^2$. Then we have upper bounds
\be\label{Jbr}
|J(\pi,\beta,\beta g;r,\beta k)|\le C r/\sinh(g r),
\ee
\be\label{Jbk}
|J(\pi,\beta,\beta g;r,\beta k)|\le C_{\varepsilon}\exp(-(\pi-\varepsilon) |k|),\ \ \ \varepsilon >0,
\ee
where the constants $C$ and $C_{\varepsilon}$ are independent of $\beta\in(0,\beta_0]$, with
\be
\beta_0\equiv \min (\pi/4,\pi/2g). 
\ee
Finally, the limit function satisfies
\be\label{Fbr}
|F(g;r,2k)|\le C r/\sinh(g r),
\ee 
\be\label{Fbk}
|F(g;r,2k)|\le C_{\varepsilon}\exp(-(\pi-\varepsilon) |k|),\ \ \ \varepsilon >0.
\ee 
 \end{proposition}
\begin{proof}
For $|\re g/2|\le R\in[1,\infty)$, $|\im r/2|\le \rho\in[0,\pi/2)$ and $t$ real, we obtain from Prop.~B.1 the majorization
\be\label{Gmaj}
\left|\frac{G(\pi,\beta;(t+\de r-i\beta g)/2)}{G(\pi,\beta;(t+\de r+i\beta g)/2)}\right|\le C(g, \rho)\left|\exp\left(-g \ln  2\cosh((t+\de r)/2)\right)\right|,
\ee
with $C$ continuous on $\cS_R\times [0,\pi/2)$ and $\beta\in(0,(\pi-2\rho)/4R]$. Letting next $\re g\in [\epsilon,2R]$ with $\epsilon\in(0,1]$, and restricting~$k\in\C$ to a strip $|\im k|\le \epsilon'$ with $\epsilon'<\epsilon$, the integrand in~\eqref{Jbeta} is $O(\exp (\epsilon'-\epsilon)t)$ as $t\to\infty$, with the implied constant independent of $\beta$. Thus we can invoke the dominated convergence theorem to justify the interchange of limits, so that the first assertion readily follows.

To prove \eqref{Jbr} we use~\eqref{Gmaj}, yielding
\be
|J(\pi,\beta,\beta g;r,\beta k)|\le C(g,0)  \int_{\R}dt \,\frac{1}{\prod_{\de=+,-}[2\cosh((t+\de r)/2)]^{g}}.
\ee
From the proof of Prop.~4.3 in~\cite{HR15} we then deduce~\eqref{Jbr}. Finally, to show~\eqref{Jbk}, we first write
\be
e^{\eta k}J(\pi,\beta,\beta g;r,\beta k)=\frac{1}{2}\int_{-\infty}^{\infty}dt\prod_{\de =+,-}\frac{G(\pi,\beta;(t+\de r-i\beta g)/2)}{G(\pi,\beta;(t+\de r+i\beta g)/2)}\exp(i(t-i\eta)k).
\ee
Now for $|\eta|\le \pi-\varepsilon$ we deduce from the bound~\eqref{Gmaj} with $t$ and $r$ swapped that we can shift the $t$-contour by $\eta$, after which we easily get the estimate 
\be
|e^{\eta k}J(\pi,\beta,\beta g;r,\beta k)|\le C(g,0)\int_{-\infty}^{\infty} dt|4\cosh((t+i\eta+r)/2)\cosh((t+i\eta-r)/2)|^{-g}.
\ee
Clearly, this entails the bound~\eqref{Jbk}. Finally, since the bounds~\eqref{Jbr} and~\eqref{Jbk} do not depend on~$\beta$, they extend to the limit function~$F$, so the proof is complete.
\end{proof}

It transpires from the proof that the uniform bounds \eqref{Jbr} and~\eqref{Jbk} can be extended to suitably restricted complex $g$, $r$ and~$k$. However, for our next aim in this section, namely to obtain two product formulas for the limit function $F(g;r,2k)$, the bounds suffice to invoke the dominated convergence theorem, as will become clear shortly.

The first product formula follows directly from Theorem~2.4 when we use \eqref{Jbr} and Appendix~B.

\begin{theorem}
Letting $g\in (0,\infty)$ and $r,s,k\in\R$, we have
\be\label{Fpr1}
F(g;r,2k)F(g;s,2k)=\frac{1}{2}\int_0^{\infty}dtF(g;t,2k)\frac{[2\sinh t]^{2g}}{\prod_{\de_1,\de_2=+,-}[2\cosh((t+\de_1 r+\de_2 s)/2)]^g}.
\ee
\end{theorem}
\begin{proof}
When we substitute  
\be\label{limsub}
a_+=\pi,\ \  a_-=\beta,\ \ b=\beta g,\ \ (x,y,z,v)=\beta (r,s,t,k),
\ee
in \eqref{prform}, then the limits~\eqref{JF}, \eqref{GGlim} and~\eqref{wlim} give rise to~\eqref{Fpr1}, provided we interchange the $\beta\to 0$ limit and the integration. To control this interchange, we need to restrict the coupling~$g$ to $[1,\infty)$. Then we can first use Prop.~B.1 and Prop.~B.2 to conclude that the product of the $w$-function and kernel function is bounded for $t\in(0,\infty)$, uniformly for $\beta$ small enough. Now it suffices to appeal to the uniform bound~\eqref{Jbr} to obtain the desired dominating function in $L^1((0,\infty))$. Therefore~\eqref{Fpr1} follows for $g\ge 1$. 

In order to lift the $g$-restriction, we note that~\eqref{JFlim} implies that for fixed $t,k\in\R$ the function $F(g;t,2k)$ on the right-hand side of~\eqref{Fpr1} is real-analytic in~$g$ for $g> 0$. From~\eqref{JFlim} it is also clear that it is bounded for $t,k$ varying over~$\R$ and~$g$ in a complex neighborhood of the positive real axis.   For $g$ in such a neighborhood, the remaining hyperbolic quotient in the integrand decays exponentially as $t\to\infty$, so for $g\in(0,1)$ the product formula~\eqref{Fpr1} follows by analytic continuation.
\end{proof}

The second product formula results from Theorem~2.5 by using~\eqref{Jbk} and Appendix~C.

\begin{theorem}
Letting $g\in (0,\infty)$ and $r,p,q\in\R$, we have
\be\label{Fpr2}
F(g;r,2p)F(g;r,2q)  =  \frac{1}{8\pi}\int_0^{\infty}dk    F(g;r,2k)
 \frac{\prod_{\de_1,\de_2,\de_3=+,-}\Gamma((g+i\de_1 p+i\de_2 q+i\de_3 k)/2)}{\Gamma(g)^2\prod_{\de =+,-}\Gamma( i\de k)\Gamma(g+ i\de k)}.
\ee
\end{theorem}
\begin{proof}
We first rewrite the integrand in~\eqref{prformalt} by using (cf.~\eqref{w} and \eqref{c}):
\be
w(2a-b;v)= \prod_{\de=+,-}G(ia+\de v)G(ia-ib+\de v).
\ee
Then we substitute $a_+=\pi,a_-=\beta,b=\beta g$, so that the right-hand side of~\eqref{prformalt} equals the $v$-integral of
\be\label{integ}
 J(\beta g;x,v)\frac{G(i\pi/2+i\beta/2-i\beta g)^2 \prod_{\de=+,-}G(i\pi/2+i\beta/2 +\de v)G(i\pi/2+i\beta/2-i\beta g+\de v)}{2\prod_{\de_1,\de_2,\de_3=+,-}G(i\pi/2+i\beta/2+(\de_1 t+\de_2 u+\de_3 v-i\beta g)/2)},
\ee
with $G(z)=G(\pi,\beta;z)$.

Substituting
\be
x=r,\ \ (t,u,v)=\beta(p,q,k),
\ee
we can rewrite~\eqref{prformalt} as
\begin{multline}
J(\pi,\beta,\beta g;r,\beta p)J(\pi,\beta,\beta g;r,\beta q)  = \frac{1}{2} \int_0^{\infty}dk J(\pi,\beta,\beta g;r,\beta k)
\\  
  \times\frac{\cG(-i g)^2 \prod_{\de=+,-}\cG(\de k)\cG(-i g+\de k)}{4\pi\prod_{\de_1,\de_2,\de_3=+,-}\cG((\de_1 p+\de_2 q+\de_3 k-i g)/2)},
\end{multline}
where $\cG(z)=\cG(\beta;z)$, cf.~\eqref{defcG}. When we now formally use  the limits~\eqref{JF} and~\eqref{cGlim}, then we obtain~\eqref{Fpr2}. The pertinent interchange of limits can be readily justified by using the uniform bounds \eqref{cGbo}--\eqref{mcGsp} and~\eqref{Jbk} to obtain an~$L^1((0,\infty))$-function dominating the pointwise convergence.
\end{proof}

In the special cases $g=1/2$ and $g=1$ the product formula \eqref{Fpr2} was first obtained by Mizony in a somewhat different form. More precisely, in Section 3 of \cite{Miz76} he considers the function
\be
\phi(r,x)\equiv {}_2F_1(1/4+ir/2,1/4-ir/2;1;-\sinh^2 x).
\ee
Specializing \eqref{FHypergRep} to $g=1/2$, and using the quadratic transformation in Eq.~15.8.18 in \cite{Dig10} for ${}_2F_1$ and the reflection equation for the gamma function to rewrite the resulting expression, we obtain
\be
\phi(r,x)=\frac{2}{\pi}\cosh(\pi r)F(1/2;x,2r).
\ee
Then Mizony's product formula
\be\label{MizProdF}
\phi(r,x)\phi(s,x)=\frac{1}{2\pi}\int_0^\infty a(r,s,t)\phi(t,x)|c(t)|^{-2}dt,
\ee
with $|c(t)|^{-2}=\pi t\tanh(\pi t)$ and $a(r,s,t)$ given by the proposition on page 5 of \cite{Miz76}, is readily seen to amount to \eqref{Fpr2} for $g=1/2$ (up to a factor 16). In Section 4, he also presents the corresponding product formula for the function
\be
\phi(r,x)\equiv\frac{\sin(rx)}{r\sinh x}=\frac{2}{\pi}\frac{\sinh(\pi r)}{r}F(1;x,2r),
\ee
and a direct computation reveals that it is equivalent to \eqref{Fpr2} with $g=1$. (The expression for $|c(r)|^{-2}$ on page 11 of \cite{Miz76} should read $|c(r)|^{-2}=r^2$, as can be inferred from the $c$-function definition on page 3.)

We mention that we can use the bounds in Appendix~C in a similar way as in the above proof to arrive at a representation of $F(g;r,2k)$ in terms of the $\Gamma$-function. Specifically, combining~\eqref{Jsym} and~\eqref{Jbeta}, we first obtain the alternative $J$-representation
\be
J(\pi,\beta,\beta g;r,\beta k)=\frac{\cG(-ig)^2}{8\pi}\int_{\R}ds\,\frac{\exp(isr)}{\prod_{\de_1,\de_2=+,-}\cG((\de_1 s+\de_2 k-ig)/2)}.
\ee
Taking $\beta\to 0$, we now get
\be\label{Frep2}
F(g;r,2k)=\frac{1}{8\pi \Gamma(g)^2}\int_{\R}ds \exp(isr)\prod_{\de_1,\de_2=+,-}\Gamma((g+i\de_1 s+i\de_2 k )/2).
\ee
(This second representation corresponds to (4.47) in~\cite{R11}, whereas the first one given by \eqref{JFlim} corresponds to (4.48).)

We proceed to obtain two distinct limits of Theorem~3.1 involving the joint two-particle eigenfunction
\be\label{F2}
F_2(g;r,k)=\exp(i(r_1+r_2)(k_1+k_2)/2)F(g;r_1-r_2,k_1-k_2).
\ee
(This function coincides with the function $F_2(\lambda;t,u)$ we employed in~\cite{HR15}, cf.~Eq.~(64) in~{\it loc.~cit.}) The first one is obtained by choosing parameters
\be\label{subpar}
a_+=\pi,\ a_-=\beta,\ b=\beta g,
\ee
and changing variables
\be\label{sub1}
x=r,\ y=\beta k/2,
\ee
in \eqref{intEq} and then taking $\beta$ to zero.

\begin{theorem}
Letting $(g,r,k)\in (0,\infty)\times\R^2\times \R^2$, we have  
\be\label{intEqnr1}
\int_{\R^2} dz \left(\frac{4\sinh^2(z_1-z_2) }{\prod_{j,l=1,2}[2\cosh(r_j-z_l)]}\right)^gF_2(g;z,k)
=2\mu_0(g;k)  F_2(g;r,k),
\ee
where
\be\label{defmu0}
\mu_0(g;k)\equiv \frac{\prod_{j=1}^2\prod_{\de=+,-}\Gamma((i\de k_j+g)/2)}{4\Gamma(g)^2}.
\ee
\end{theorem}
\begin{proof}
With the above substitutions, the  $\beta\to 0$ limit of the left-hand side of \eqref{intEq} yields the left-hand side of~\eqref{intEqnr1}, cf.~the proof of Theorem~4.2. For the eigenvalue~\eqref{defmu} we combine the substitutions with the reparametrization~\eqref{defcG} to get
\be
\cG(-ig)^2/4\prod_{j=1}^2\prod_{\de=+,-}\cG((\de k_j-ig)/2).
\ee
Invoking \eqref{cGlim}, \eqref{J2J}, \eqref{F2}  and \eqref{JF}, we see that \eqref{intEqnr1} results.
\end{proof}

For the second limit we start from \eqref{intEqalt}. 
We choose once more parameters given by~\eqref{subpar}, but now change variables
\be\label{sub2}
y=r,\ z=\beta p/2,\ x=\beta k/2.
\ee
Then it is clear from~\eqref{JF} that the function $J_2(b;y,x)$ converges to $F_2(g;r,k)$ for $\beta \to 0$. But we need a multiplicative renormalization for the right-hand side to have a finite limit. Specifically, from \eqref{defmu} we see that the substitutions entail
\be\label{muren}
a_-^{-1}G(ia-ib)^2\mu(2a-b;y)\to \pi \prod_{j=1}^2\prod_{\de=+,-}G(\pi,\beta;\de r_j-i\beta g/2),
\ee
so that we can invoke \eqref{GGlim}. Thus we arrive at the right-hand side of the identity in the following theorem.

\begin{theorem}
Letting $(g,r,k)\in (0,\infty)\times\R^2\times \R^2$, we have  
\begin{multline}\label{intEqnr2}
\frac{1}{16\pi \Gamma(g)^2}\int_{\R^2} dp \frac {\prod_{j,l=1,2}\prod_{\de=+,-}\Gamma((i\de (k_j-p_l)+g)/2)}{\prod_{\de=+,-}\Gamma(i\de (p_1-p_2)/2)\Gamma(i\de (p_1-p_2)/2+g)}F_2(g;r,p)
\\
= \frac{2\pi}{[4\cosh (r_1) \cosh (r_2)]^g}F_2(g;r,k).
\end{multline}
\end{theorem}
\begin{proof}
It remains to handle the left-hand side of~\eqref{intEqalt}, multiplied by the factor~$a_-^{-1}G(ia-ib)^2$, cf.~the renormalization \eqref{muren}. This can be done just as in the proof of Theorem~4.3, and the result is~\eqref{intEqnr2}.
\end{proof}

We proceed to obtain the nonrelativistic counterparts of the results from Sections~2 and~3 that involve Hilbert space analysis. To start with, we define an auxiliary unitary transform
\be
\cF_{\beta}(g)\, : \, \cC\subset L^2((0,\infty),dk) \to L^2((0,\infty),dr),\ \ \ \beta g\in(0,\pi+\beta),
\ee
\be
(\cF_{\beta}(g)\psi)(r)\equiv \left(\frac{1}{2\pi}\right)^{1/2} \int_0^\infty \rF(\pi,\beta,\beta g;r,\beta k)\psi(k)dk,\ \ \psi\in\cC,
\ee
cf.~\eqref{cF}--\eqref{cFE}. Using~\eqref{rFJ} and~\eqref{defcG}, we readily obtain
\begin{multline}
\rF(\pi,\beta,\beta g;r,\beta k)=\frac{2}{\cG(\beta;-ig)}w(\pi,\beta,\beta g;r)^{1/2}J(\pi,\beta,\beta g;r,\beta k)
\\
\times \left(\prod_{\de=+,-}\cG(\beta;\de k)\cG(\beta;\de k-ig)\right)^{1/2}.
\end{multline}
From this we deduce that for all $(g,r,k)\in (0,\infty)^3$ we have
\begin{multline}\label{rF0}
 \rF_0(g;r,k)\equiv \lim_{\beta \to 0} \rF(\pi,\beta,\beta g;r,\beta k) 
 \\
  =  2\Gamma(g)(2\sinh r)^gF(g;r,2k)\left(\prod_{\de=+,-}\Gamma(i\de k)\Gamma(i\de k+g)\right)^{-1/2}.
\end{multline}

The operator on $\cC$ defined by
\be
(\cF_{0}(g)\psi)(r)\equiv \left(\frac{1}{2\pi}\right)^{1/2} \int_0^\infty \rF_0(g;r, k)\psi(k)dk,\ \ \psi\in\cC,
\ee
gives rise to a unitary transform 
\be
\cF_{0}(g)\, : \,   L^2((0,\infty),dk) \to L^2((0,\infty),dr),\ \ \ g> 0.
\ee
 (It equals the sine transform for $g=1$ and its $g\to 0$ limit is the cosine transform.) Although this assertion will cause no surprise, a complete proof is not immediate from our results. However, the unitarity of $\cF_0(g)$ follows by specialization from known unitarity properties of the Jacobi function transform~\cite{Koo84}.
 
Next, we define a family of integral operators on~$L^2((0,\infty))$ by setting
\be\label{cJ1}
(\cJ_t(g)\psi)(r)\equiv \int_0^{\infty} ds\, \frac{ w_0(g;r)^{1/2} w_0(g;s)^{1/2}}{\prod_{\de_1,\de_2=+,-}[2\cosh((t+\de_1 s+\de_2 r)/2)]^g} \psi(s), 
\ee
\be
w_0(g;r)\equiv (2\sinh r)^{2g},\ \ \ r>0,
\ee
where $g>0$, $t\ge 0$ and the implied logarithm is chosen real. Hence the kernel of the integral operator is positive for all $(g,r,s,t)\in(0,\infty)^3\times [0,\infty)$. Since it has exponential decay for $s\to \infty$, the integral is absolutely convergent, but it is not clear that the image function is square-integrable. In fact, however, the family consists of bounded self-adjoint operators satisfying
\be
[\cJ_{t_1}(g),\cJ_{t_2}(g)]=0,\ \ \ t_1,t_2 \ge 0.
\ee
This is an obvious consequence of the following theorem.

\begin{theorem}
Let $g>0$ and $t\ge 0$. Then the operator~$\cJ_t(g)$ is bounded, and the operator $\cF_0(g)^*\cJ_t(g)\cF_0(g)$ on $L^2((0,\infty),dk)$ acts as multiplication by the function $2F(g;t,2k)$. 
\end{theorem}
\begin{proof}
When we swap $t$ and $s$ in~\eqref{Fpr1} and use~\eqref{rF0}, the resulting identity can be written
\begin{multline}\label{rF0id1}
2F(g;t,2k)\rF_0(g;r,k)=
\\
\int_0^{\infty}ds \, \frac{  w_0(g;r)^{1/2} w_0(g;s)^{1/2}}{\prod_{\de_1,\de_2=+,-}[2\cosh((t+\de_1 s+\de_2 r)/2)]^g} \rF_0(g;s,k).
\end{multline}
Integrating this with $\psi(k)$, $\psi\in L^2((0,\infty))$, we deduce
\be
2\cF_0(g)F(g;t, 2\,\cdot)\psi=\cJ_t(g)\cF_0(g)\psi.
\ee
From this the assertion is plain.
\end{proof}

To continue, we define a dual family of integral operators on $L^2((0,\infty))$ by
\begin{multline}
(\hat{\cJ}_q(g)\phi)(k)\equiv 
 \int_0^{\infty}dp 
 \,  \hat{w}_0(g;k)^{1/2} \hat{w}_0(g;p)^{1/2}\phi(p)
\\
\times \prod_{\de_1,\de_2,\de_3=+,-}\Gamma((g+i\de_1 p+i\de_2 q+i\de_3 k)/2) , 
 \end{multline}
 \be
\hat{w}_0(g;k)\equiv
1/4\pi \Gamma(g)^2\prod_{\de=+,-}\Gamma(i\de k)\Gamma(i\de k+g),
\ee
 where $g>0$, $q\ge 0$ and the positive square root is taken. When we combine~\eqref{Gabo} and~\eqref{Garefl}, we deduce exponential decay of the kernel of the integral operator for $p\to \infty$, so the integral is absolutely convergent. As in previous cases, it is not obvious that the image function is square-integrable. Once more, however, the family actually consists of bounded self-adjoint operators satisfying
\be
[\hat{\cJ}_{q_1}(g),\hat{\cJ}_{q_2}(g)]=0,\ \ \ q_1,q_2 \ge 0,
\ee
as is immediate from our next theorem.

\begin{theorem}
Let $g>0$ and $q\ge 0$. Then the operator~$\hat{\cJ}_q(g)$ is bounded, and the operator $\cF_0(g)\hat{\cJ}_q(g)\cF_0(g)^*$ on $L^2((0,\infty),dr)$ acts as multiplication by the function $2F(g;r,2q)$. 
\end{theorem}
\begin{proof}
Swapping $k$ and $p$ in~\eqref{Fpr2} and using~\eqref{rF0}, we deduce
\begin{multline}\label{rF0id2}
2F(g;r,2q)\rF_0(g;r,k)=
\int_0^{\infty}dp \, \hat{w}_0(g;k)^{1/2} \hat{w}_0(g;p)^{1/2} \rF_0(g;r,p)
\\
\times \prod_{\de_1,\de_2,\de_3=+,-}\Gamma((g+i\de_1 p+i\de_2 q+i\de_3 k)/2).
\end{multline}
Integrating this with $\phi(r)$, $\phi\in L^2((0,\infty))$, we obtain
\be
2\cF_0(g)^*F(g;\cdot, 2q)\phi=\hat{\cJ}_q(g)\cF_0(g)^*\phi,
\ee
and so the theorem follows.
\end{proof}

Passing to the 2-particle case, we define a unitary operator by (continuous extension of)
\be\label{cF20}
\cF_{0,2}(g)
\, :\, \cC_2= C_0^{\infty}(G_2)\subset   L^2(G_2) \to  L^2(G_2),\ \ \ g>0, 
\ee
where
\be\label{cFE20}
(\cF_{0,2}(g)\psi)(r)\equiv \frac{1}{2\pi} \int_{G_2}\rF_{0,2}(g;r,k)\psi(k)dk,\ \ \ \
\psi\in\cC_2,\ \ r\in G_2,
\ee
with
\be
\rF_{0,2}(g;r,k)\equiv \exp(i(r_1+r_2)(k_1+k_2)/2)\rF_0(g;r_1-r_2,(k_1-k_2)/2).
\ee
Now we define two integral operators on $L^2(G_2)$:
\be
(\cJ_2(g)\phi)(r)\equiv
\int_{G_2} dz\,  \frac{w_0(g;r_1-r_2)^{1/2}w_0(g;z_1-z_2)^{1/2} }{\prod_{j,l=1,2}[2\cosh(r_j-z_l)]^g} \phi(z), 
\ee
 \begin{multline}
(\hat{\cJ}_2(g)\psi)(k)\equiv  \int_{G_2} dp\, \hat{w}_0(g;k_1-k_2)^{1/2}\hat{w}_0(g;p_1-p_2)^{1/2}\psi(p)
\\
\times  \prod_{j,l=1,2}\prod_{\de=+,-}\Gamma((i\de (k_j-p_l)+g)/2) ,  
\end{multline}
where $\phi,\psi\in L^2(G_2)$ and $g>0$. We are now prepared for the last theorem of this section.

\begin{theorem}
Letting $g>0$, the operator  $\cF_{0,2}(g)^*\cJ_2(g)\cF_{0,2}(g)$ on $L^2(G_2,dk)$ acts as multiplication by the positive bounded function $\mu_0(g;k)$ given by \eqref{defmu0}. Also,
 the operator $\cF_{0,2}(g)\hat{\cJ}_2(g)\cF_{0,2}(g)^*$ on $L^2(G_2,dr)$ acts as multiplication by~$4\pi/[4\cosh (r_1) \cosh (r_2)]^g $.
\end{theorem}
\begin{proof}
Recalling \eqref{F2} and \eqref{rF0}, this can be read off from \eqref{intEqnr1} and \eqref{intEqnr2}.
\end{proof}

We conclude this section by deriving nonrelativistic counterparts of the identities~\eqref{id1} and~\eqref{iden1}. First, we obtain alternative versions of the product formulas in Theorems~4.2 and~4.3, using the unitarity of the transform~$\cF_0(g)$. To this end we multiply~\eqref{rF0id1} by~$\rF_0(g;v,k)$ and integrate over~$k$. Now we use~\eqref{rF0}, rewritten as
\be\label{rF0alt}
F(g;r,2k)=\rF_0(g;r,k)/4\sqrt{\pi}\Gamma(g)^2w_0(g;r)^{1/2}\hat{w}_0(g;k)^{1/2}.
\ee
As a result, we obtain the identity
\begin{multline}\label{idnr1}
\frac{w_0(g;r)^{1/2}w_0(g;s)^{1/2}w_0(g;t)^{1/2}}{\prod_{\de_1,\de_2=+,-}[2\cosh((r+\de_1 s+\de_2 t)/2)]^g}=
\\
\frac{1}{4\pi^{3/2}\Gamma(g)^2}\int_0^{\infty} dk\, \hat{w}_0(g;k)^{-1/2}\rF_0(g;r,k)\rF_0(g;s,k)\rF_0(g;t,k).
\end{multline}
Likewise, we multiply~\eqref{rF0id2} by~$\rF_0(g;r,v)$ and integrate over~$r$. Using again~\eqref{rF0alt}, we arrive at
\begin{multline}\label{idnr2}
\hat{w}_0(g;k)^{1/2}\hat{w}_0(g;p)^{1/2}\hat{w}_0(g;q)^{1/2}\prod_{\de_1,\de_2,\de_3=+,-}\Gamma((g+i\de_1 k+i\de_2 p+i\de_3 q)/2) =
\\
\frac{1}{4\pi^{3/2}\Gamma(g)^2}\int_0^{\infty} dr\, w_0(g;r)^{-1/2}\rF_0(g;r,k)\rF_0(g;r,p)\rF_0(g;r,q).\end{multline}

Turning to the 2-particle case, we  first use \eqref{rF0alt} to rewrite \eqref{intEqnr1} as
\be
\int_{G^2} dz  \frac{w_0(g;r_1-r_2)^{1/2}w_0(g;z_1-z_2)^{1/2} }{\prod_{j,l=1,2}[2\cosh(r_j-z_l)]^g} \rF_{0,2}(g;z,k)
=\mu_0(g;k)  \rF_{0,2}(g;r,k).
\ee
Now we multiply this by~$\overline{\rF_{0,2}}(g;s,k)/4\pi^2$ and integrate~$k $ over~$G_2$. This yields the identity
\be\label{idennr1}
\frac{w_0(g;r_1-r_2)^{1/2}w_0(g;s_1-s_2)^{1/2}}{\prod_{j,l=1,2}[2\cosh(r_j-s_l)]^g}=\frac{1}{4\pi^2}
\int_{G_2}dk\, \mu_0(g;k)\rF_{0,2}(g;r,k)\overline{\rF_{0,2}}(g;s,k).
\ee

Similarly, using \eqref{rF0alt} to rewrite \eqref{intEqnr2} as
\begin{multline}  
 \int_{G^2} dp \, \hat{w}_0(g;k_1-k_2)^{1/2}\hat{w}_0(g;p_1-p_2)^{1/2}\rF_{0,2}(g;r,p)
 \prod_{j,l=1,2}\prod_{\de=+,-}\Gamma((i\de (k_j-p_l)+g)/2) 
 \\
= \frac{4\pi}{[4\cosh (r_1) \cosh (r_2)]^g}\rF_{0,2}(g;r,k),
\end{multline}
multiplying this by~$\overline{\rF_{0,2}}(g;r,q)/4\pi^2$ and then integrating~$r $ over~$G_2$, we get
\begin{multline}\label{idennr2}
\hat{w}_0(g;k_1-k_2)^{1/2}\hat{w}_0(g;q_1-q_2)^{1/2}
\prod_{j,l=1,2}\prod_{\de=+,-}\Gamma((i\de (k_j-q_l)+g)/2) 
\\
=\frac{1}{\pi}\int_{G_2}dr\, \frac{\rF_{0,2}(g;r,k)\overline{\rF_{0,2}}(g;r,q)}{[4\cosh (r_1) \cosh (r_2)]^g}.
\end{multline}

Finally, just as the 2-particle identity~\eqref{iden1} leads to the reduced identity~\eqref{id1} upon using sum and difference variables, we can rederive \eqref{idnr1}/\eqref{idnr2} from \eqref{idennr1}/\eqref{idennr2} by invoking the $F$-representations \eqref{Frep2}/\eqref{JFlim} and then using \eqref{rF0alt}, respectively. We note that this yields a nontrivial check on the various constants involved.

\begin{appendix}

\section{The hyperbolic gamma function}\label{AppA}
In this appendix we review previously known properties of the hyperbolic gamma function $G(a_+,a_-;z)$ we have occasion to use in Sections 2 and~3. (More details can be found in~\cite{R97} and Appendix~A of \cite{R99}.) Throughout the paper we choose the parameters~$a_+$ and~$a_-$ positive. As a rule, the dependence on these parameters  shall be suppressed when they are supposed to be fixed.  

The hyperbolic gamma function was introduced in Section III~A of \cite{R97} as the unique minimal solution of one of the two A$\De$Es
\be\label{GADeEs}
\frac{G(z+ia_\de/2)}{G(z-ia_\de/2)}=2c_{-\de}(z),\ \ \ \de=+,-,
\ee
that has modulus $1$ for real $z$ and satisfies $G(0)=1$. It is not obvious, but true, that the other one is satisfied as well. Furthermore, $G(z)$ is meromorphic in~$z$, and it has neither poles nor zeros for $z$ in the strip  
\be\label{strip}
S\equiv \{z\in\C \mid |\im (z)|<a\},\ \ \ a=(a_++a_-)/2.
\ee
Hence we have
\be\label{Gg}
G(z)=\exp(ig(z)),\ \ \ \ z\in S,
\ee
with $g(z)$ holomorphic in $S$. Explicitly, $g(z)$ has the integral representation
\be\label{ghyp}
g(a_{+},a_{-};z) =\int_0^\infty\frac{dy}{y}\left(\frac{\sin 2yz}{2\sinh(a_{+}y)\sinh(a_{-}y)} - \frac{z}{a_{+}a_{-} y}\right),\ \ \ \ z\in S.
\ee
This clearly implies that the hyperbolic gamma function has the following properties:
\be\label{Grefl}	
G(-z) = 1/G(z),\ \ \ ({\rm reflection\ equation}),
\ee
\be\label{modinv}
G(a_-,a_+;z) = G(a_+,a_-;z),\ \ \  ({\rm modular\ invariance}),
\ee
\be\label{sc}
G(\lambda a_+,\lambda a_-;\lambda z) = G(a_+,a_-;z),\quad \lambda\in(0,\infty),\ \ \ ({\rm scale\ invariance}),
\ee
\be\label{Gconj}
\overline{G(a_+,a_-;z)}=G(a_+,a_-;-\overline{z}).
\ee

From Appendix A in \cite{R99} we recall that $G(z)$ can be written as 
\be\label{GE}
G(z)=E(z)/E(-z),
\ee
with $E(z)$ an entire function with zeros located only at the points
\be\label{Ezs}
z=ia+ip_{kl},\ \ \ a=(a_+ +a_-)/2,\ \ \  k,l\in\N\equiv \{ 0,1,2,\ldots\}, 
\ee
where
\be\label{pkl}
p_{kl}\equiv ka_++la_-.
\ee
Moreover, the order of these zeros  (denoted by $\cO(kl)$), equals the number of distinct pairs $(m,n)\in\N^2$ such that $p_{mn}=p_{kl}$. In particular, for $a_+/a_-\notin\Q$ all zeros are simple. Clearly, \eqref{GE} entails that $G$ has the same zero set as $E$ and poles of order $\cO(kl)$ located solely at
\be\label{Gpo}
z=-ia-ip_{kl},\ \ \ k,l\in\N.
\ee

We also recall that the asymptotic behaviour of $G(z)$ for $\re z\to\pm\infty$ is given by
\be\label{Gas}
G(z)=\exp(\mp i(\chi+\alpha z^2/4))(1+O(\exp(-r|\re z|))),\ \ \ \re z\to\pm\infty,
\ee
where
\be
\chi\equiv \frac{\pi}{24}\left(\frac{a_+}{a_-}+\frac{a_-}{a_+}\right),
\ee
the decay rate can be any positive number satisfying
\be\label{r}
r<\alpha\min(a_+,a_-),
\ee
and the implied constant is uniform for $\im z$ varying over compact subsets of $\R$.

Finally, in Section \ref{Sec2} we have occasion to make use of the Fourier transform formula from Proposition C.1 in \cite{R11}. More specifically, let $\mu,\nu\in\C$ be such that
\be
-a<\im\mu<\im\nu<a,
\ee
and assume that $y\in\C$ satisfies
\be
|\im y|<\im(\nu-\mu)/2.
\ee
Then the pertinent formula is given by
\begin{multline}\label{Fform}
\int_\R dz\exp(i\alpha zy)\frac{G(z-\nu)}{G(z-\mu)}\\
=\sqrt{a_+a_-}\exp(i\alpha y(\mu+\nu)/2)G(ia+\mu-\nu)\prod_{\de=+,-}G(\de y-ia+(\nu-\mu)/2).
\end{multline}

\section{Uniform bounds on $G$-function ratios}\label{AppB}

In this appendix and the next one, we reconsider and improve limits involving the hyperbolic gamma function. We shall use the results to control the nonrelativistic limit in Section~4.

We begin by recalling from Subsection~III~A in~\cite{R97} the limit
\be\label{GGlim}
\lim_{\beta\downarrow 0}\frac{G(\pi, \beta;z+i\beta u)}{G(\pi, \beta;z+i\beta d)}
=\exp((u-d)\ln (2\cosh z)). 
\ee
Here, we have $u,d\in\R$, the logarithm is real-valued for~$z$ real, and~\eqref{GGlim} holds true uniformly  for~$z$ varying over compact subsets of the cut plane
\be\label{Ccut}
\C(\pi)\equiv \C \setminus \{ \pm i[\pi/2,\infty)\},
\ee
cf.~{\it loc.~cit.}~(3.91). 
Now the definition~\eqref{w} of the $w$-function entails
\be
w(\pi,\beta,\beta g;z)=\frac{G(\pi,\beta;z-i\pi /2+i\beta(g-1/2))}{G(\pi,\beta;z-i\pi
/2+i\beta(-1/2))}\cdot \frac{G(\pi ,\beta;z+i\pi /2+i\beta(1/2))}{G(\pi,\beta;z+i\pi /2
+i\beta(1/2-g))},
\ee
so from~\eqref{GGlim} we deduce
\be\label{wlim}
\lim_{\beta \downarrow 0}w(\pi,\beta ,\beta  g; z)=\exp (2g\ln
(2\sinh  z)), 
\ee
where~$g$ is real, the limit is uniform on compacts of the open right half plane and the logarithm is real-valued for $z>0$.
 
In this appendix we obtain bounds on the $G$-ratios occurring in~\eqref{GGlim}/\eqref{wlim} that are uniform for $\beta$ sufficiently small and  for $\re z\in\R$/$z\in(0,\infty)$, so that we can appeal to the dominated convergence theorem for $(\re z)$-integrals  that involve ratios of the above type. More specifically, we aim for uniform bounds that involve the limit function. To see what this entails for~\eqref{GGlim}, it is helpful to inspect a simple special case. Specifically, let us take $u=v-1/2$ and $d=v+1/2$ in the pertinent $G$-ratio. Then by~\eqref{GADeEs} it equals $1/2\cosh (z+i\beta v)$. Thus for a bound  
\be
\frac{1}{|\cosh (z+i\beta v)|}<\frac{C}{ |\cosh z|},
\ee
to hold with $C>0$ independent of~$\re z\in\R$ and~$\beta$ in a sufficiently small interval~$(0,\beta_0]$, we need to require $\im z\in(-\pi/2,\pi/2)$ and then choose~$\beta_0>0$ such that $|\im z|+\beta_0 |\im v|<\pi/2$. 

This example goes to show that the following proposition cannot be much improved.

\begin{proposition}
Let~$u,d\in\cS_R$, where
\be\label{defcS}
\cS_R\equiv \{ v\in \C\mid |\re v|\le R\},\ \ \ R\ge 1, 
\ee
 and let~$z\in\C$ satisfy $|\im z|\le \rho$, $\rho\in[0,\pi/2)$.
 Choosing~$\beta_0>0$ such that
\be\label{beta0}
\beta_0 R\le (\pi-2\rho)/4,
\ee
we have for all~$\beta\in (0,\beta_0]$ a bound 
\be\label{GGb}
 \left|\frac{G(\pi, \beta;z+i\beta u)}{G(\pi, \beta;z+i\beta d)}\right|\le C(u,d,\rho) |\exp((u-d)\ln (2\cosh z))|,
 \ee
 where $C$ is a positive continuous function on~$\cS_R^2\times [0,\pi/2)$ and the logarithm is real for~$z$ real.
\end{proposition}
\begin{proof}
We begin by pointing out that
  we have
\be\label{udres}
\im (z+i\beta u), \im (z+i\beta d)\in (-\pi/2,\pi/2),
\ee
by virtue of the $\beta$-restriction. Next, we show that we can use~\eqref{GADeEs} to reduce the case where $|\re d|$ and/or $|\re u|$ is larger than 1/2 to the case $u,d\in\cS_{1/2}$. 

Indeed, consider first the assumption $\re d\in (N-1/2,N+1/2]$ for an integer $N\ge 1$. 
Then we can write
\be
 G(\pi,\beta;z+i\beta d)=G(\pi,\beta;z+i\beta (d-N))2^N\prod_{k=0}^{N-1}\cosh(z+i\beta(d-k-1/2)),
 \ee
 so that $d-N\in\cS_{1/2}$. 
Now in view of the restrictions~\eqref{udres} and $|\im z|\le \rho$, we can use the bound
\be
1/\left| \prod_{k=0}^{N-1}\cosh(z+i\beta(d-k-1/2))\right|\le C_+ |\cosh(z)|^{-N},
\ee
with $C_+>0$ depending on~$N$, $\im d$ and~$\rho$, but not on~$\beta$ and~$\re z$, so as to reduce consideration to~$\cS_{1/2}$.

Likewise, assuming $\re d\in [-N-1/2,-N+1/2)$, we can use 
\be
\left| \prod_{k=0}^{N-1}\cosh(z+i\beta(d+k+1/2))\right|\le C_- |\cosh(z)|^N,
\ee
with $C_-$ independent of~$\beta$ and~$\re z$ to replace the denominator $G(\pi,\beta;z+i\beta d)$ by $G(\pi,\beta;z+i\beta (d+N))$, with $d+N\in\cS_{1/2}$. Of course, the numerator~$G(\pi,\beta;z+i\beta u)$ can be treated analogously in case $|\re u|>1/2$.

Accordingly, we assume from now on
\be\label{uds}
u,v\in\cS_{1/2}.
\ee 
We first use the integral representation resulting from~\eqref{Gg} and~\eqref{ghyp} to write
\be
\frac{G(\pi, \beta;z+i\beta u)}{G(\pi, \beta;z+i\beta d)}=\exp   \int_0^\infty\frac{dy}{y}\left(\frac{\sinh (\beta (d-u)y)\cos(2yz+i\beta (d+u)y)}{\sinh(\pi y)\sinh(\beta y)} + \frac{u-d}{\pi y}\right).
\ee
Using the identity (cf.~(3.21) in~\cite{R97})
\be
\ln (2 \cosh z)= \int_0^\infty \frac{dy}{y}\left(\frac{1}{\pi y}-\frac{\cos 2yz}{\sinh \pi y}\right),\ \ \ |\im z|<\pi /2,
\ee
we now deduce
\be
\frac{G(\pi, \beta;z+i\beta u)}{G(\pi, \beta;z+i\beta d)}
\exp ((d-u)\ln (2 \cosh z))=\exp \int_0^\infty\frac{dy}{y\sinh \pi y} I(y),
\ee
where we have set
\be
I(y)\equiv
 \frac{\sinh \beta (d-u)y}{\sinh\beta y} \cos(2yz+i\beta (d+u)y) +(u-d) \cos 2yz.
\ee
As a consequence, the proposition follows when we can prove a bound
\be\label{intb}
\left| \int_0^\infty\frac{dy}{y\sinh \pi y} I(y)\right|\le c(u,d,\rho),
\ee
with $c$ continuous on~$\cS_{1/2}^2\times [0,\pi/2)$.

To this end, we write the integral as a sum of integrals
\be
I_1\equiv \int_1^\infty\frac{dy}{y\sinh \pi y} I(y),\ \ 
I_2\equiv \int_0^1\frac{dy}{y\sinh \pi y} I(y).
\ee
 In the first integral we write
\bea
I(y) & = & \left( \frac{\sinh \beta (d-u)y}{\sinh\beta y} -(d-u)\right)\cos(2yz+i\beta (d+u)y)
\\  \nonumber
&  &  +(d-u)(\cos(2yz+i\beta (d+u)y)- \cos 2yz).
\eea
Now from our assumption~\eqref{uds} we deduce a bound
\be
\left|\frac{\sinh (d-u)x}{\sinh x}-(d-u)\right|<c(d-u),\ \ \ \forall x>0,
\ee
with $c(v)$ continuous on $\cS_1$; also, using~\eqref{beta0} we obtain
\bea
| \cos(2yz+i\beta (d+u)y)|  &  \le & \exp( y(2|\im z|+\beta|\re (d+u)|))
\\  \nonumber
  &  \le & \exp (y(\pi/2+\rho)).
  \eea
Hence we get
\bea\label{I1b}
|I_1|  &  \le  & \int_1^\infty\frac{dy}{y\sinh \pi y}\big( [c(d-u)+|d-u|]\exp (y(\pi/2+\rho)) 
\\  \nonumber
  &  & +|d-u|\exp (2y|\im z|)\big)
\le   c_1(u,d, \rho),
  \eea
where $c_1$ is continuous on $\cS_{1/2}^2\times [0,\pi/2)$.
 
In the second integral~$I_2$ we use
\be
\frac{\sinh \beta (d-u)y}{\sinh \beta y}=(d-u)\big(1+ O(\beta^2y^2)\big),
\ee
and
\be
\cos (2yz+i\beta (d+u)y)=\big(\cos 2yz  -i\beta (d+u)y\sin 2yz\big)\big(1+ O(\beta^2y^2)\big),
\ee
so that we wind up with
\be\label{ints}
I(y)=-i\beta (d^2-u^2)y\sin 2yz +O(\beta^2 y^2),\ \ \ \beta y\to 0.
\ee
Here, the implied constant can be chosen uniformly for $u$ and $d$ varying
over $\C$-compacts and $\im z$ varying over bounded intervals.
  From this  we readily deduce
\be\label{I2b}
|I_2|\le c_2(u,d,\im z),
\ee
with $c_2$ continuous on $\C^2\times\R$. 

Combining the estimates~\eqref{I1b} and~\eqref{I2b} yields~\eqref{intb}, so the proposition follows.
\end{proof}

Consider next the weight function. Again, it is illuminating to consider a special case: Letting~$N\in\N$ with $N>1$, it follows from~\eqref{GADeEs} that we have an explicit evaluation
\be\label{wNev}
w(\pi,\beta,N\beta;z)=2^{2N}\sinh (z)^2\prod_{k=1}^{N-1}\sinh(z+ik\beta)\sinh(z-ik\beta).
\ee
 Hence we  should not aim for a bound $C\sinh(z)^{2N}$ with $C$ uniform for $z>0$ and $\beta$ sufficiently small, since such a bound is not valid near $z=0$. On the other hand, this snag is not present for a bound of the form $C\sinh (Nz)^2$, and the $\beta\to 0$ limit is majorized by the latter function for $C=2^{2N}$. Indeed, a more general inequality holds true:
 \be\label{shest}
 \sinh(z)^a\le \sinh (az),\ \ \ z>0,\ \ \  a\ge 1.
 \ee
(Its proof is straightforward, cf.~(101)--(102) in~\cite{HR15}.) 

For $a<1$ it is manifestly false that the function~$\sinh(z)^a$ is majorized by~$C\sinh (az)$ for all~$z>0$. Therefore the $g$-restriction in the following proposition cannot be relaxed.

\begin{proposition}
Letting $\beta\in (0,1]$, we have an inequality
\be\label{west}
w(\pi,\beta,g\beta;z)< C(g)\sinh (gz)^2,
\ee
where $g\ge 1$, $z>0$ and $C$ is a positive continuous function on $[1,\infty)$.
\end{proposition}
\begin{proof}
On account of~\eqref{wNev} and~\eqref{shest}, the bound~\eqref{west} holds true for~$g=N$. Our proof for the general case involves a lot more effort. We begin by using a couple of formulas from Section~V~A in~\cite{R97}, where the above weight function was studied in considerable detail. 

The first formula is {\it loc.~cit.}~(5.15), which implies
\be\label{wf1}
w(\pi,\beta,(g+1)\beta;z)=4\sinh(z+ig\beta)\sinh(z-ig\beta)w(\pi,\beta,g\beta;z).
\ee
(To be sure, \eqref{wf1} can be derived directly from~\eqref{GADeEs}.) From this it is easy to see that we need only show~\eqref{west} for $g\in[1,2]$.

Next, we introduce a new weight function $w_+$ by setting
\be\label{wplus}
w(\pi,\beta,g\beta;z)=4\sinh(z)^2  w_+(\pi,\beta,g\beta;z),
\ee
cf. {\it loc.~cit.}~(5.14). Then it suffices to prove the bound
\be\label{wpb}
w_+(\pi,\beta,g\beta;z)<C\exp(2(g-1)z),\ \ \forall (\beta,g,z)\in (0,1]\times [1,2]\times  (0,\infty).   
\ee
To this end we invoke the integral representation
\be\label{wprep}
w_+(\pi,\beta,g\beta;z)=\exp\left(-2\int_0^\infty \frac{dy}{y}\left( Q(\beta,g;y)\cos 2yz +\frac{1-g}{\pi y}
\right)\right),
\ee
where
\be\label{defQ}
Q(\beta,g;y)\equiv \frac{\sinh (g-1)\beta y}{\sinh \beta y}\cdot \frac{\cosh (\pi- g\beta)y}{\sinh \pi y},
\ee
cf. {\it loc.~cit.}~(5.35) and (5.36). (This representation can be derived by using~\eqref{Gg} and~\eqref{ghyp}.) In order to exploit this explicit formula, we need the auxiliary integral
\be
f(z)\equiv \frac{1}{\pi}\int_0^\infty \frac{dy}{y^2}(1-\cos 2yz)=\frac{2z}{\pi}\int_0^\infty \frac{du}{u^2}(1-\cos u).
\ee
Integrating by parts in the integral over~$(\epsilon, R)$ and then taking~$\epsilon\to 0$ and $R\to\infty$, we arrive at $\int_0^\infty du \sin u/u=\pi/2$, so that $f(z)=z$. We can therefore rewrite~\eqref{wprep} as
\be\label{wprepnew}
w_+(\pi,\beta,g\beta;z)=\exp(2(g-1)z)\exp\left(-2\int_0^\infty \frac{dy}{y}\left( Q(\beta,g;y)+\frac{1-g}{\pi y}
\right)\cos 2yz \right).
\ee

Comparing~\eqref{wprepnew} to \eqref{wpb}, we conclude that it remains to prove
\be\label{brem}
\int_0^\infty \frac{dy}{y}\left( Q(\beta,g;y)+\frac{1-g}{\pi y}
\right)\cos 2yz > c, 
\ee
with $c\in\R$ independent of $(\beta,g,z)\in (0,1]\times [1,2]\times (0,\infty)$. We begin by showing this holds for the integral over $(0,1)$. Then we can use
\be
\left( Q+\frac{1-g}{\pi y}\right)= (g-1)\left( \coth (\pi y)-\frac{1}{\pi y}\right)+O(\beta y),\ \ \ \beta y\to 0,
\ee
where the implied constant can be chosen uniformly for $(\beta,g)\in (0,1] \times [1,2]$. Clearly, this entails
\be
\left| \int_0^1 \frac{dy}{y}\left( Q+\frac{1-g}{\pi y}
\right)\cos 2yz\right|\le c_1,
\ee
with $c_1$ independent of $(\beta,g,z)\in(0,1]\times [1,2]\times (0,\infty)$.

Obviously uniform boundedness is also true for the integral
\be
\int_1^\infty \frac{dy}{y^2}\cos 2yz,
\ee
contributing to the integration over $(1,\infty)$ in~\eqref{brem}. Hence we are done when we can show
\be\label{bfin}
\int_1^\infty \frac{dy}{y}Q(\beta,g;y) \cos 2yz >c_2,
\ee
with $c_2$ independent of $(\beta,g,z)\in(0,1]\times [1,2]\times (0,\infty)$. This last step of our proof is the most arduous one. The difficulty is that the integral does not converge absolutely for $\beta =0$, since $Q$ then reduces to $\coth(\pi y)$, cf.~\eqref{defQ}. 

Our method to get around this last obstacle hinges on $Q$ being a monotonically decreasing function of $y$. This feature is not obvious at face value, but can be gleaned from the logarithmic derivative
\bea
Q'/Q  &  = & (g-1)\beta\coth (g-1)\beta y -\beta \coth \beta y
\\ \nonumber
 & & +(\pi -g\beta)\tanh (\pi -g\beta)y-\pi \coth \pi y.
 \eea
Indeed, for all $(\beta,g,y)\in(0,1]\times [1,2]\times (0,\infty)$, the difference on the second line is clearly negative, whereas a moment's thought shows the first difference is nonpositive. (This follows in particular from $x\coth x$ being increasing on $(0,\infty)$.)

For a fixed $z>0$, consider now the smallest zero $y_0(z)\in(1,\infty)$ of $\cos 2yz$ for which $\sin (2y_0(z)z)=-1$. We assert that the integral of $y^{-1}Q(\beta,g;y)\cos 2yz $ over $(y_0(z),\infty)$ is positive. To see this, note that by monotonicity of $Q/y$ the integral over an interval between two successive zeros of the function $\cos 2yz$ on which it is positive is larger than that over the next such interval, on which it is negative. From this our assertion follows.

It remains to show that the integral over the interval $(1,y_0(z))$ is uniformly bounded below for $(\beta,g,z)\in(0,1]\times [1,2]\times (0,\infty)$. First we put  
\be
\Lambda \equiv Q(\beta,g;1),
\ee
and note that $\Lambda$ is uniformly bounded above. Next we fix attention on the $z$-values
\be
z\in (0,3\pi/4) \Rightarrow y_0(z)=3\pi/4z.
\ee
Then the interval length is bounded above by $3\pi/4z$ and on this interval the integrand is bounded below by $-\Lambda$. The integral is therefore bounded below by $-3\pi \Lambda/4z$, but with this crude bound we cannot exclude a negative divergence as $z$ goes to~0.

However, as soon as $z<\pi/4$, the function $\cos 2yz$ is positive for $y\in[1,\pi/4z)$ and negative for $y\in (\pi/4z,3\pi/4z)$. Now the latter interval has length $\pi/2z$, but on it the integrand is actually bounded below by $-\Lambda \cdot 4z/\pi$. Hence the integral over this interval is bounded below by $-2
\Lambda$.

The upshot is that the pertinent integral is uniformly bounded below for $z\in (0,3\pi/4)$. Choosing next $z\ge 3\pi/4$, we need only use once more that the interval length on which $\cos 2yz$ is negative equals $\pi/2z$, together with $Q(\beta,g;y)/y<\Lambda$ for $y>1$, to deduce that the contribution from the leftmost negative interval is bounded below by $-\Lambda \cdot \pi/2z\ge -2\Lambda/3$. This completes the proof that the integral is uniformly bounded below, so the proposition follows. 
\end{proof}

\section{Uniform bounds on the $G\to\Gamma$ limit}\label{AppC}

In Subsection~III~A of~\cite{R97} it is shown that the function
\be\label{defH}
H(s;z)\equiv G(1,s;s z+i/2)\exp[ iz\ln (2\pi s)-2^{-1}\ln (2\pi)],\ \ \ s>0,
\ee
converges to $1/\Gamma(iz+1/2)$ as $s\to 0$, uniformly for~$z$ in $\C$-compacts, cf.~(3.72) in~{\it loc.~cit.} This involves the product function
\be\label{defP}
P(s;z)\equiv H(s;z)\Gamma(iz+1/2),
\ee
and its integral representation
\be\label{intP}
P(s;z)=\exp I(s;z),\ \ \ |\im z|<1/2+1/s,
\ee
where
\be\label{defI}
I(s;z)\equiv i\int_0^{\infty}f_s(t)[\sin(2zt)-2z\sinh (t)]dt,
\ee
\be\label{fs}
f_s(t)\equiv \frac{\exp(-t/s)}{2t\sinh (t)\sinh(t/s)},
\ee
cf.~the proof of Prop.~III.6 in~{\it loc.~cit.} Indeed, from this representation it is easy to check 
\be\label{Plim}
\lim_{s\to 0}P(s;z)=1,
\ee
uniformly on $\C$-compacts.

Just as in Appendix~B, we now supplement this limit with bounds that are uniform for~$\re z\in\R$ and~$s$ small enough, with a view to invoke dominated convergence for $G$-integrals occurring in the main text, cf.~Section~4. We begin by estimating~$|P(s;z)|$. 

\begin{proposition}
 For all $z\in\C$ with $|y|=|\im z|\le R\in[1,\infty)$ and $s\in(0,1/R]$, we have
\be\label{Pbo}
|P(s;z)| =\exp(\pi sy|x|)\exp(m_P(s;z)),\ \ \ z=x+iy,
\ee
where 
\be\label{mP}
|m_P(s;z)|<c(R)\ln(1+|x|)+d(R),
\ee
with $c$ and $d$ positive continuous functions on $[1,\infty)$. 
\end{proposition}
\begin{proof}
Letting $|y|\le R$ and $s\in(0,1/R] $, we have
\be\label{PK}
|P(s;z)|=\exp K(s;z),
\ee
where
\be\label{defK}
K(s;z)\equiv \int_0^{\infty}f_s(t)D(z,t)dt,\ \ \ 
D(z,t)\equiv 2y\sinh (t)- \sinh(2yt)\cos(2xt),
\ee
as is easily verified from~\eqref{intP}--\eqref{fs}. Now we write
\be
m_P(s;z)\equiv K(s;z)-\pi sy|x|=K(s;z)+sy\int_0^{\infty}\frac{dt}{t^2}(\cos 2xt -1),
\ee
and telescope the right-hand side as
\be
 \sum_{j=1}^2\int_0^1I_j(s;z,t)dt + \int_1^{\infty}f_s(t)D(z,t)dt +sy\int_1^{\infty}\frac{dt}{t^2}(\cos 2xt -1),
\ee
with
\be
I_1\equiv  f_s(t)L_1(y,t),\ \ \ L_1\equiv 2y \sinh t-\sinh 2 yt,
\ee
\be
I_2\equiv  L_2(s;y,t)(\cos 2xt-1)/t,\ \ \ L_2\equiv sy/t -tf_s(t)\sinh 2yt.
\ee

We proceed to bound the four summands of $m_P(s;z)$. First,  we have
\be
L_1(y,t)=\frac{t^3}{3}(y-4y^3)+O(t^5),\ \ t\to 0,
\ee
whence we deduce
\be\label{b1}
\left|\int_0^1 I_1dt\right|<c_1(R),\ \ \  s\in (0,1/R],\ \ |y|\le R,
\ee
with $c_1$ continuous on $[1,\infty)$.
 To estimate the second summand, we use the mean value theorem to rewrite~$L_2$ as
 \be
L_2=\frac{sy}{t}\left(1-e^{-t/s}\frac{q(t)q(t/s)}{q(2yt)}\right)=-sy\partial_t 
\left(e^{-t/s}\frac{q(t)q(t/s)}{q(2yt)}\right)_{t=t'},\ \ \ q(v)\equiv \frac{v}{\sinh v},
\ee
with $t'\in(0,t)$.  From this we deduce
\be
|L_2(s;y,t)|<c_2(R),\ \ \  s\in (0,1/R],\ \ |y|\le R,\ \ t\in(0,1),
\ee
with $c_2$ continuous on $[1,\infty)$.
Hence we obtain
\be\label{b2}
 \left|\int_0^1 I_2dt\right|  < 2c_2(R) \int_0^{|x|}\frac{du}{u} \sin^2u<2c_2(R)(c+\ln (1+|x|)),\ \ \  s\in (0,1/R],\ \ |y|\le R,
\ee
with $c=\int_0^{\pi/2}du\sin^2u/u$, say.

For the third summand we use
\be
|D(z,t)|< 2R\sinh t+\sinh 2Rt,\ \ \ |y|\le R,\ \ \ t>1, 
\ee
to infer
\be\label{b3}
\left|\int_1^{\infty}f_s(t)D(z,t)dt\right| <c_3(R),\ \ \  s\in (0,1/R],\ \ |y|\le R,
\ee
with $c_3$ continuous on $[1,\infty)$.
Finally, we clearly have
\be
\left|sy\int_1^{\infty}\frac{dt}{t^2}(\cos 2xt -1)\right|<2,\ \ \ s\in (0,1/R],\ \ |y|\le R.
\ee
Combining this with the bounds~\eqref{b1}, \eqref{b2} and~\eqref{b3}, the proposition readily follows.
\end{proof}

Next, we bound~$|\Gamma(iz+1/2)|$, $z=x+iy$, for $y<1/2$, by using the representation
\be\label{Garep}
\Gamma(iz+1/2)=(2\pi)^{1/2}\exp\left(\int_0^{\infty}\frac{dt}{t}\left(\frac{e^{-2izt}}{2\sinh t}-\frac{1}{2t}+ize^{-2t}\right)\right),
\ee
cf.~(A37) in~\cite{R97}.

\begin{proposition}
 For all $z\in\C$ with $y=\im z<1/2 $, we have
\be\label{Gabo}
|\Gamma(iz+1/2)| = \exp(-\pi |x|/2-y\ln (1+|x|)+m_{\Gamma}(z)),\ \ \ z=x+iy,
\ee
where 
\be
|m_{\Gamma}(z)|<d(y),
\ee
with $d(y)$ a continuous function on $(-\infty,1/2)$. 
\end{proposition}
\begin{proof}
We follow the proof of the previous proposition, both in spirit and in notation.
Letting $y<1/2$, we have
\be\label{GaK}
|\Gamma(iz+1/2)|=(2\pi)^{1/2}\exp K(z),
\ee
where
\be\label{defK2}
K(z)\equiv \int_0^{\infty}\frac{dt}{t}D(z,t),\ \ \ 
D(z,t)\equiv \frac{e^{2yt}}{2\sinh t}\cos 2xt-\frac{1}{2t}-ye^{-2t},
\ee
as is clear from~\eqref{Garep}. Now we write
\begin{multline}
  K(z)+\pi |x|/2+y\ln (1+|x|)=K(z)-\frac{1}{2}\int_0^{\infty}\frac{dt}{t^2}(\cos 2xt -1)+y\ln (1+|x|)
\\
=  \sum_{j=1}^2\int_0^1I_j(z,t)dt +
 \int_1^{\infty}\frac{dt}{t}D(z,t) 
 -\frac{1}{2}\int_1^{\infty}\frac{dt}{t^2}(\cos 2xt -1),
\end{multline}
where
\be
I_1\equiv  t^{-1}L_1(y,t),\ \ \ L_1\equiv \frac{e^{2yt}}{2\sinh t}-\frac{1}{2t}-y e^{-2t},
\ee
\be
I_2\equiv  y\ln (1+|x|)+L_2(y,t)(\cos 2xt-1)/t,\ \ \ L_2\equiv \frac{e^{ 2yt}}{2\sinh t}-\frac{1}{2t}.
\ee

We proceed to bound the four terms. First we  observe
\be
L_1(y,t)=t(y^2+2y-1/12)+O(t^2),\ \ t\to 0.
\ee
From this we obtain
\be\label{e1}
\left|\int_0^1 I_1dt\right|<d_1(y), 
\ee
with $d_1$ continuous on~$\R$.
To handle the second term, we write $L_2$ as
\be
L_2(y,t)=y+tr(y,t),
\ee
and note that we then have
\be
|r(y,t)|<c_2(y),\ \ \ t\in(0,1),
\ee
with $c_2$ continuous on~$\R$. Using the bound
\be
\left|\ln (1+a)+\int_0^a\frac{du}{u}(\cos 2u -1)\right|<c,\ \ a\ge 0,
\ee
whose proof is straightforward, we now obtain
\be\label{e2}
\left| \int_0^1 I_2dt \right|  < d_2(y),
\ee
with $d_2$ continuous on $\R$.

In order to bound the third term, we use
\be
|D(z,t)|< \frac{e^{2yt}}{2\sinh t}+\frac{1}{2t}+|y|e^{-2t},\ \ y<1/2,\ \ \ t>1, 
\ee
to get
\be\label{e3}
\left|\int_1^{\infty}\frac{dt}{t}D(z,t)\right| <d_3(y), 
\ee
with $d_3$ continuous on~$(-\infty,1/2)$.
Finally, we clearly have
\be
\frac{1}{2}\left|\int_1^{\infty}\frac{dt}{t^2}(\cos 2xt -1)\right|<1, 
\ee
so together with~\eqref{e1}, \eqref{e2} and~\eqref{e3}, this yields the proposition.
\end{proof}

 For the applications we have in mind, it is   expedient to switch from $H(s;z)$~\eqref{defH} to
\be\label{defcG}
\cG(\beta;z)\equiv H(\beta/\pi;z+i/2)=G(\pi,\beta;i\pi/2+i\beta/2+\beta z)\exp(iz\ln (2\beta)-2^{-1}\ln (4\pi\beta)).
\ee
(Here we used the scaling relation~\eqref{sc}.) The last result of this paper is now readily obtained by combining the two previous propositions.

\begin{proposition}
We have
\be\label{cGlim}
\lim_{\beta\to 0}\cG(\beta;z)=1/\Gamma(iz),
\ee
uniformly for $z$ varying over $\C$-compacts.  For all $z\in\C$ with $\im z\in[-R,0)$, $R\in [1,\infty)$, and~$\beta\in(0,\pi/ R]$, we have
\be \label{cGbo}
|\cG(\beta;z)| =\exp[(\beta(\im z+1/2)+\pi/2)|\re z|]
 \exp(m_{\cG}(\beta;z)), 
\ee
where
\be\label{mcG}
|m_{\cG}(\beta;z)|< \gamma(\im z)\ln (1+|\re z|)+\de(\im z),
\ee
with $\gamma$ and $\de$ positive continuous functions on~$(-\infty,0)$. Finally, for all~$k\in\R$  and~$\beta\in(0,1/2\pi R]$, $R\in [1,\infty)$, we have
\be\label{cGbosp}
|\cG(\beta;k)| =\exp[(\beta/2+\pi/2)|k|]
 \exp(m_{\cG}(\beta;k)), 
\ee
where
\be\label{mcGsp}
|m_{\cG}(\beta;k)|< C_1\ln (1+|k|)+C_2.
\ee
\end{proposition}
\begin{proof}
Clearly, \eqref{cGlim} follows from~\eqref{Plim}. Likewise, \eqref{cGbo}--\eqref{mcG} are clear from combining the previous two propositions.  To obtain~\eqref{cGbosp}--\eqref{mcGsp}, however, we cannot use the last one, since now we have $y=1/2$. 

On the other hand, from the reflection equation we infer
\be\label{Garefl}
|1/\Gamma(ik)|=(k\sinh(\pi k)/\pi)^{1/2},\ \ k\in\R,
\ee
and when we combine this with~\eqref{Pbo}--\eqref{mP} we deduce~\eqref{cGbosp}--\eqref{mcGsp}. 
\end{proof}

\end{appendix}

\section*{Acknowledgments}
We would like to thank Tom Koornwinder for his interest and a number of helpful comments, and for drawing our attention to \cite{Miz76}. An application of the product formula~\eqref{Fpr2} can be found in his recent preprint \cite{Koo16}.

\bibliographystyle{amsalpha}

\begin{thebibliography}{R03III}
\addcontentsline{toc}{section}{References}

\bibitem[vdB06]{vdB06} F.~J.~van de Bult, \emph{Ruijsenaars' hypergeometric function and the modular double of ${\mathcal U}_q(sl_2(\C))$}, Adv.~Math. {\bf 204} (2006), 539--571.
 
\bibitem[BRS07]{BRS07} F.~J.~van de Bult, E.~M.~Rains and J.~V.~Stokman, \emph{ Properties of generalized univariate hypergeometric functions}, Commun.~Math.~Phys. {\bf 275} (2007), 37--95.

\bibitem[Dig10]{Dig10} Digital Library of Mathematical Functions, Release date 2010-05-07, National Institute of Standards and Technology, http://dlmf.nist.gov.

\bibitem[HR14]{HR14} M.~Halln\"as and S.~N.~M.~Ruijsenaars, \emph{Joint eigenfunctions for the relativistic Calogero-Moser Hamiltonians of hyperbolic type. I. First steps}, Int.~Math.~Res.~Not.~(2014), no.~16, 4400--4456.

\bibitem[HR15]{HR15} M.~Halln\"as and S.~N.~M.~Ruijsenaars, \emph{A recursive construction of joint eigenfunctions for the hyperbolic  nonrelativistic Calogero-Moser Hamiltonians}, Int.~Math.~Res.~Not.~(2015), no.~20, 10278--10313.
 
\bibitem[Koo84]{Koo84} T.~H.~Koornwinder, \emph{Jacobi functions and analysis on noncompact semisimple Lie groups}, in: Special functions: group theoretical aspects and applications (R.~A.~Askey, T.~H.~Koornwinder and W.~Schempp, Eds.), Mathematics and its applications, Reidel, Dordrecht, 1984, pp.~1--85.

\bibitem[Koo16]{Koo16} T.~H.~Koornwinder, \emph{Dual addition formulas associated with dual product formulas}, arXiv:1607.06053.

\bibitem[Miz76]{Miz76} M.~Mizony, \emph{Alg\`ebres et noyaux de convolution sur le dual sph\'erique d'un groupe de Lie semi-simple, non compact et de rang $1$}, Publ.~D\'ep.~Math.~(Lyon) {\bf 13} (1976), 1--14.

\bibitem[R94]{R94} S.~N.~M.~Ruijsenaars, \emph{Systems of Calogero-Moser type}, in: Proceedings of the 1994 Banff summer school ``Particles and fields" (G. Semenoff and L. Vinet, Eds.), CRM series in mathematical physics, Springer, New York, 1999, pp. 251--352.

\bibitem[R97]{R97} S.~N.~M.~Ruijsenaars, \emph{First-order analytic difference equations and integrable quantum systems}, J.~Math.~Phys.~{\bf 38} (1997), 1069--1146.

\bibitem[R99]{R99} S.~N.~M.~Ruijsenaars, \emph{A generalized hypergeometric function satisfying four analytic difference equations of Askey-Wilson type}, Commun.~Math.~Phys.~{\bf 206} (1999), 639--690.
 
\bibitem[R03II]{R03II} S.~N.~M.~Ruijsenaars, \emph{ A generalized hypergeometric function II. Asymptotics and $D_4$ symmetry},  Commun. Math. Phys. {\bf 243} (2003), 389--412.

\bibitem[R03III]{R03III} S.~N.~M.~Ruijsenaars, \emph{ A generalized hypergeometric
function III. Associated Hilbert space transform}, 
Commun. Math. Phys. {\bf 243} (2003) 413--448.
 
 \bibitem[R07]{R07} S.~N.~M.~Ruijsenaars, \emph{Quadratic transformations for a function that generalizes ${}_2F_1$ and the Askey-Wilson polynomials}, in Askey Festschrift issue (Bexbach 2003 Proceedings), Ramanujan J.~{\bf 13} (2007), 339--364.
 
\bibitem[R11]{R11} S.~N.~M.~Ruijsenaars, \emph{A relativistic conical function and its Whittaker limits}, SIGMA~{\bf 7} (2011), 101, 54 pages.

\bibitem[R13]{R13} S.~N.~M.~Ruijsenaars, \emph{On positive Hilbert-Schmidt operators}, Int.~Eq.~Oper.~Theory {\bf 75} (2013), 393--407.

\end{thebibliography}

\end{document}